\documentclass[12pt,a4paper, oneside,reqno]{article}

\usepackage[utf8]{inputenc}
\usepackage[english]{babel}

\usepackage[left=1in,right=1in,top=1in,bottom=1in]{geometry}

\usepackage{amsfonts,amssymb,amsmath,amsthm}

\usepackage{comment}
\usepackage[symbol]{footmisc} 
\usepackage{longtable}
\usepackage{enumerate}
\usepackage{cite}
\usepackage{hyperref}
\usepackage{tkz-graph}

\usepackage{fancyhdr} 
\theoremstyle{plain}
\newtheorem{theorem}{Theorem}
\newtheorem{lemma}[theorem]{Lemma}
\newtheorem{proposition}[theorem]{Proposition}

\newtheorem{definition}[theorem]{Definition}
\newtheorem{corollary}[theorem]{Corollary}

\usetikzlibrary{arrows}

\newtheorem*{theoremA1}{Theorem A1}
\newtheorem*{theoremA2}{Theorem A2}
\newtheorem*{theoremG1}{Theorem G1}
\newtheorem*{theoremG2}{Theorem G2}

\usepackage{fouriernc} 

\begin{document}


\bigskip

\noindent{\Large
The algebraic and geometric classification of \\
right alternative superalgebras}
 \footnote{
The  authors would like to thank the SRMC (Sino-Russian Mathematics Center in Peking University, Beijing, China) for its hospitality and excellent working conditions, where some parts of this work were done. The    work is supported by 
FCT   2023.08031.CEECIND and UID/00212/2025.}

 \bigskip

\begin{center}

 {\bf
Hani Abdelwahab\footnote{Department of Mathematics, 
 Mansoura University,  Mansoura, Egypt; \ haniamar1985@gmail.com}, 
   Ivan Kaygorodov\footnote{CMA-UBI, University of  Beira Interior, Covilh\~{a}, Portugal; \     Moscow Center for Fundamental and Applied Mathematics, Moscow,   Russia;   \    kaygorodov.ivan@gmail.com}   \&
   Abror Khudoyberdiyev \footnote{Institute of Mathematics Academy of
Sciences of Uzbekistan, Tashkent, Uzbekistan; National University of Uzbekistan, Tashkent, Uzbekistan; \ khabror@mail.ru}  
}

\end{center}

\noindent {\bf Abstract:}
{\it  
The algebraic and geometric classifications of complex 
$3$-dimensional right alternative superalgebras are given.
As a byproduct, we have the algebraic and geometric classification of the variety of $3$-dimensional 
$\mathfrak{perm}$, 
binary $\mathfrak{perm}$, 
associative, 
binary associative,
$\big(-1,1\big)$-,
and binary $\big(-1,1\big)$-superalgebras.

}

 \bigskip 


\bigskip


	 \bigskip

 
\tableofcontents

\newpage

\section*{Introduction}

The algebraic classification (up to isomorphism) of algebras of dimension $n$ of a certain variety
defined by a family of polynomial identities is a classic problem in the theory of non-associative algebras 
$\big($see \cite{ ahk, G62, akl, als, ah, ikm}$\big)$.
There are many results related to the algebraic classification of small-dimensional algebras in different varieties of
associative and non-associative algebras.
 Deformations and geometric properties of a variety of algebras defined by a family of polynomial identities have been an object of study since the 1970's  
 $\big($see \cite{ben,BC99,GRH, GRH3,ikm, kz, FKS,FKS25}  and references in \cite{k23,l24,MS}$\big)$. 
 Burde and Steinhoff constructed the graphs of degenerations for the varieties of    $3$-dimensional and $4$-dimensional Lie algebras~\cite{BC99}. 
 Grunewald and O'Halloran studied the degenerations for the variety of $5$-dimensional nilpotent Lie algebras~\cite{GRH}.

One of the most important generalizations of associative algebras is the class of alternative algebras.
They are defined by the following identities 
\begin{equation}    \label{ident}
    \big(x,y,z\big) \ = \ - \ \big(x,z,y\big)\mbox{ \ and \ }\big(x,y,z\big) \ = \ \big(y,z,x\big),
\end{equation}
where $\big(x,y,z\big)=(xy)z-x(yz).$
The main example of (non-associative) alternative algebras is the octonion algebra discovered in  the XIX century.
The variety of alternative algebras is under certain consideration now 
$\big($see, for example, \cite{C24,C25,G24,AA,hnt} and reference therein$\big)$.
The variety of algebras defined as the first identity from \eqref{ident} is called right alternative algebras. 
Albert proved that every semisimple right alternative algebra of characteristic not two is alternative \cite{alb2} and 
after that, 
Thedy showed that a finite dimensional right alternative algebra without a nil ideal is alternative \cite{thedy77}.
But, unfortunately, it was proven that the variety of right alternative algebras does not admit the  Wedderburn principal theorem \cite{thedy78}.
In the infinite case we have a different result:
Mikheev constructed an example of an infinite-dimensional simple right alternative algebra that is not alternative over any field \cite{mikh} and 
it  was finally established that a simple right alternative algebra must be either alternative or nil \cite{skos1}.
Right alternative algebras with some additional identities were studied in papers of Isaev, Pchelintsev and others \cite{isaev}.
For example, Isaev gave the negative answer for the Specht problem in the variety of right alternative algebras \cite{isaev}.
Some combinatorial properties of the variety of right alternative algebras were studied by Umirbaev in \cite{ualbay85}.
Right alternative unital bimodules over the matrix algebras of order $\geq 3$ and over Cayley algebra were studied by 
Murakami,  Pchelintsev,  Shashkov, and Shestakov in \cite{MPS,PSS}.  
On the other side, there is an  interest in the study of 
nilpotent and solvable right alternative algebras 
$\big($see, for example, \cite{skos3,ser76,ser13}$\big)$.
So, Pchelintsev proved an analog of the well-known Zhevlakov theorem: any right alternative Malcev-admissible  nil algebra of bounded index over a field of characteristic zero is solvable \cite{sergey}.
The study of super-generalizations of principal varieties of non-associative algebras is under a certain interest now 
$\big($see \cite{ahk,k23,KM,   B25, BBE, DES, BM, BK, 1,2,3,4,5} and references therein$\big)$.
The study of simple right alternative superalgebras has a very big progress in recent papers of Pchelintsev and Shashkov (see,
for example, see their survey \cite{seroleg} and the reference therein).

Gainov proved that each $3$-dimensional alternative algebra is associative \cite{G62}; 
the classification of all  
$3$-dimensional right alternative  and all $4$-dimensional nilpotent right alternative algebras were given in  \cite{akl,ikm}. 
The aim of our present work is to classify all $3$-dimensional right alternative superalgebras (Theorems A1 and A2) and their principal subvarieties. 
Based on the received result, we will construct the geometric classification of $3$-dimensional right alternative superalgebras  (Theorems G1 and G2).

\section{The algebraic classification of superalgebras}

\begin{definition}
A superalgebra $\rm{A}$ over
a field $\mathbb{C}$ with a ${\mathbb Z}_{2}$-grading $\rm{A=A}_{0}\rm{\oplus A}_{1}$ is called right alternative if its multiplication satisfies
the right alternative surep-identity: 
\begin{equation*}
\left( x,y,z\right) =-\left( -1\right) ^{\left\vert y\right\vert \left\vert
z\right\vert }\left( x,z,y\right),
\end{equation*}%
\noindent
for all homogeneous elements $x,y,z$ in $\rm{A}$, where $\left\vert
x\right\vert $ denotes the degree of $x$ $(0$ for even elements, $1$ for odd
elements$)$.
\end{definition}

\subsection{Preliminaries: the method for algebraic classification}
In this section, we establish a straightforward method to obtain the
algebraic classification of the right alternative superalgebra structures of
type $\left( n,m\right) $ defined over an arbitrary Jordan superalgebra of
type $\left( n,m\right) $.

The  Jordan superproduct and supercommutator are defined as follows:

\begin{longtable}{lcl}
$x\bullet y $&$=$&$\frac{1}{2} \big(xy+\left( -1\right) ^{\left\vert x\right\vert \left\vert
y\right\vert }yx \big),$ \\
$\left[ x,y\right] $&$=$&$\frac{1}{2} \big(xy-\left( -1\right) ^{\left\vert x\right\vert
\left\vert y\right\vert }yx \big).$
\end{longtable}
The superalgebra $\left( \rm{A},\bullet \right) $ is called the
symmetrized superalgebra of $\rm{A}$ and is denoted by $\rm{A}^{+}.$ 
If $\rm{A}$ is a right alternative superalgebra, then $\rm{A}^{+}$
is a Jordan superalgebra.

Let $\mathbb{V=V}_{0}\mathbb{\oplus V}_{1}$ be a ${\mathbb Z}_{2}$-graded vector
space over a field $\mathbb{C}$. A bilinear map%
\begin{equation*}
\theta :\mathbb{V}\times \mathbb{V\to V}
\end{equation*}%
is called supercommutative if for all homogeneous elements $x,y\in \mathbb{V}$, the following graded commutativity condition holds:%
\begin{equation*}
\theta \left( x,y\right) =\left( -1\right) ^{\left\vert x\right\vert
\left\vert y\right\vert }\theta \left(y,x\right) .
\end{equation*}%
Moreover, $\theta $ is called super-skew symmetry if for all homogeneous
elements $x,y\in \mathbb{V}$, the following graded commutativity condition
holds:%
\begin{equation*}
\theta \left( x,y\right) =-\left( -1\right) ^{\left\vert x\right\vert
\left\vert y\right\vert }\theta \left( y,x\right) .
\end{equation*}

\begin{definition}
Let $\left( \rm{A},\bullet \right) $ be a Jordan superalgebra of type $\left( n,m\right) $. Define ${\rm Z}^{2}\left( \rm{A},\rm{A}\right) $ to
be the set of all super-skew-symmetric bilinear maps $\theta :\rm{A}%
\times \rm{A}\to \rm{A}$ such that:
\[
\begin{aligned}
& (x \bullet y) \bullet z \;+\; \theta(x \bullet y, z) \;+\; \theta(x, y) \bullet z \;+\; \theta(\theta(x, y), z) \\
& \quad + (-1)^{|y||z|} \Big( (x \bullet z) \bullet y \;+\; \theta(x \bullet z, y) \;+\; \theta(x, z) \bullet y \;+\; \theta(\theta(x, z), y) \Big) \\
&= x \bullet (y \bullet z) \;+\; \theta(x, y \bullet z) \;+\; x \bullet \theta(y, z) \;+\; \theta(x, \theta(y, z)) \\
& \quad + (-1)^{|y||z|} \Big( x \bullet (z \bullet y) \;+\; \theta(x, z \bullet y) \;+\; x \bullet \theta(z, y) \;+\; \theta(x, \theta(z, y)) \Big) .
\end{aligned}
\]
\noindent{}for all $x,y,z  \in \rm{A}_0 \cup \rm{A}_1$.
\end{definition}

Observe that, for $\theta \in {\rm Z}^{2}\left( \rm{A},\rm{A}\right) $,
if we define a multiplication $\ast _{\theta }$ on $\rm{A}$ by 
\begin{center}
    $x\ast_{\theta }y\ =\ x\bullet y+\theta \left( x,y\right) $ for all $x,y$ in $\rm{A}$,
\end{center}
then $\left( \rm{A},\ast _{\theta }\right) $ is a right
alternative superalgebra of type $\left( n,m\right) $. Conversely, if $%
\left( \rm{A},\cdot \right) $ is a right alternative superalgebra of
type $\left( n,m\right) $, then there exists $\theta \in {\rm Z}^{2}\left( 
\rm{A},\rm{A}\right) $ such that $\left( \rm{A},\cdot,\right) \cong \left( \rm{A},\ast _{\theta }\right) $. To see this,
consider the super-skew symmetric bilinear map $\theta :\rm{A}\times \rm{A}\to \rm{A}$ defined by $\theta \left( x,y\right)
=\left[ x,y\right] $ for all $x,y$ in $\rm{A}$. Then $\theta \in
{\rm Z}^{2}\left( \rm{A},\rm{A}\right) $ and $\left( \rm{A},\cdot
,\right) \cong \left( \rm{A},\ast _{\theta }\right) $.

\medskip

Now, let $\left( \rm{A},\cdot \right) $ be a right alternative
superalgebra and $\rm{Aut}\left( \rm{A}\right) $ be the automorphism
group of $\rm{A}$. Then $\rm{Aut}\left( \rm{A}\right) $ acts on $%
{\rm Z}^{2}\left( \rm{A},\rm{A}\right) $ by 
\begin{equation*}
\left( \theta \ast \phi \right) \left( x,y\right) =\phi ^{-1}\left( \theta
\left( \phi \left( x\right) ,\phi \left( y\right) \right) \right) ,
\end{equation*}%
for $\phi \in \text{Aut}\left( \rm{A}\right) $, and $\theta \in {\rm Z}^{2}\left( \rm{A},\rm{A}\right) $.

\begin{lemma}
Let $\left( \rm{A},\cdot \right) $ be a Jordan superalgebra and $\theta
,\vartheta \in {\rm Z}^{2}\left( \rm{A},\rm{A}\right) $. Then $\left( 
\rm{A},\ast _{\theta }\right) $ and $\left( \rm{A},\ast
_{\vartheta }\right) $ are isomorphic if and only if there is a $\phi \in 
\rm{Aut}\left( \rm{A}\right) $ with $\theta \ast \phi =\vartheta $.
\end{lemma}

Hence, we have a procedure to classify the right alternative superalgebras
of type $\left( n,m\right) $ associated with a given Jordan superalgebra $%
\left( \rm{A},\bullet \right) $ of type $\left( n,m\right) $. It
consists of three steps:

\begin{enumerate}
\item Compute ${\rm Z}^{2}\left( \rm{A},\rm{A}\right) $.

\item Find the orbits of $\rm{Aut}\left( \rm{A}\right) $ on $%
{\rm Z}^{2}\left( \rm{A},\rm{A}\right) $.

\item Choose a representative $\theta $ from each orbit and then construct
the right alternative superalgebra $\left( \rm{A},\ast _{\theta
}\right) $.
\end{enumerate}

\bigskip

Let us introduce the following notations. Let 
\begin{center}
    $e_{1},\ldots
,e_{n},e_{n+1}=f_{1},\ldots,e_{n+m}=f_{m},$
\end{center} be a fixed basis of a Jordan
superalgebra $\left( \rm{A},\bullet \right) $ of type $\left( n,m\right) $.
Define $\mathrm{\Lambda }^{2}\left( \rm{A},\mathbb{C}\right) $ to be
the space of all super skew symmetric bilinear forms on $\rm{A}$. Then 
\begin{center}$%
\mathrm{\Lambda }^{2}\left( \rm{A},\mathbb{C}\right) =\left\langle
\Delta _{i,j}:1\leq i<j\leq n+m\right\rangle \oplus \left\langle \Delta
_{i,i}:n<i\leq n+m\right\rangle,$
\end{center} where $\Delta _{i,j}$ is the super skew-symmetric bilinear form $\Delta _{i,j}:\rm{A}\times \rm{A}%
\to \mathbb{C}$\ defined by%
\begin{equation*}
\Delta _{i,j}\left( e_{l},e_{m}\right) :=\left\{ 
\begin{tabular}{ll}
$1,$ & if $\left( i,j\right) =\left( l,m\right) $ and $i\neq j,$ \\ 
$-\left( -1\right) ^{\left\vert e_{l}\right\vert \left\vert e_{m}\right\vert
},$ & if $\left( i,j\right) =\left( m,l\right) $ and $i\neq j,$ \\ 
$1,$ & if $i=l=m$ and $i=j>n,$ \\ 
$0,$ & otherwise.%
\end{tabular}%
\right. 
\end{equation*}%
 Now, if $\theta \in {\rm Z}^{2}\left( \rm{A},\rm{A}\right) $, then $%
\theta $ can be uniquely written as $\theta \left( x,y\right) =\underset{i=1}%
{\overset{n}{\sum }}B_{i}\left( x,y\right) e_{i}$ where $B_{1},B_{2},\ldots
,B_{n}$ is a sequence of super-skew-symmetric bilinear forms on $\rm{A}$%
. Also, we may write $\theta =\left( B_{1},B_{2},\ldots ,B_{n}\right).$ 
Let $\phi ^{-1}\in \text{Aut}\left( \rm{A}\right) $ be given by the
matrix $\left( b_{ij}\right) $. If 
\begin{center}
    $\left( \theta \ast \phi \right) \left(
x,y\right) =\underset{i=1}{\overset{n}{\sum }}B_{i}^{\prime }\left(
x,y\right) e_{i}$, then $B_{i}^{\prime }=\underset{j=1}{\overset{n}{\sum }}%
b_{ij}\phi ^{t}B_{j}\phi $.

\end{center}

 \newpage
\subsection{Right alternative superalgebras of type $(1,2)$}

\begin{proposition}
Let ${\rm J}$ be a $3$-dimensional Jordan superalgebra of type $\left(
1,2\right) $. Then ${\rm J}$ is isomorphic to one of the following
superalgebras:

\begin{longtable}{c c l l l l}
 
    \hline
    ${\rm J}_{01}$ & $:$ &  & $e_{1}\cdot f_{1}=f_{2}$ &  & $f_{1}\cdot f_{2}=e_{1}$ \\
    ${\rm J}_{02}$ & $:$ &  &  &  & $f_{1}\cdot f_{2}=e_{1}$ \\
    ${\rm J}_{03}$ & $:$ &  & $e_{1}\cdot f_{1}=f_{2}$ &  &  \\
    ${\rm J}_{04}$ & $:$ & $e_{1}\cdot e_{1}=e_{1}$ & $e_{1}\cdot f_{1}=\frac{1}{2} f_{1}$ &  &  \\
    ${\rm J}_{05}$ & $:$ & $e_{1}\cdot e_{1}=e_{1}$ & $e_{1}\cdot f_{1}=f_{1}$ & $e_{1}\cdot f_{2}=\frac{1}{2} f_{2}$ &  \\
    ${\rm J}_{06}$ & $:$ & $e_{1}\cdot e_{1}=e_{1}$ & $e_{1}\cdot f_{1}=f_{1}$ &  &  \\
    ${\rm J}_{07}$ & $:$ & $e_{1}\cdot e_{1}=e_{1}$ &  &  &  \\
    ${\rm J}_{08}$ & $:$ & $e_{1}\cdot e_{1}=e_{1}$ & $e_{1}\cdot f_{1}=\frac{1}{2} f_{1}$ & $e_{1}\cdot f_{2}=\frac{1}{2} f_{2}$ &  \\
    ${\rm J}_{09}$ & $:$ & $e_{1}\cdot e_{1}=e_{1}$ & $e_{1}\cdot f_{1}=f_{1}$ & $e_{1}\cdot f_{2}=f_{2}$ &  \\
    ${\rm J}_{10}$ & $:$ & $e_{1}\cdot e_{1}=e_{1}$ & $e_{1}\cdot f_{1}=\frac{1}{2} f_{1}$ & $e_{1}\cdot f_{2}=\frac{1}{2} f_{2}$ & $f_{1}\cdot f_{2}=e_{1}$ \\
    ${\rm J}_{11}$ & $:$ & $e_{1}\cdot e_{1}=e_{1}$ & $e_{1}\cdot f_{1}=f_{1}$ & $e_{1}\cdot f_{2}=f_{2}$ & $f_{1}\cdot f_{2}=e_{1}$ \\
    ${\rm J}_{12}$ & $:$ & \text{trivial} \\
    \hline
\end{longtable}

\end{proposition}

\subsubsection{The classification Theorem A1}
\begin{theoremA1}
Let ${\rm R}$ be a nontrivial $3$-dimensional right alternative
superalgebra of type $\left( 1,2\right) $. Then ${\rm R}$ is isomorphic
to one of the following superalgebras:

\begin{longtable}{llllllllllllllllllll}
${\rm R}_{01}^{\alpha \neq -1}$&$:$&
$e_{1}f_{1}=\left( 1+\alpha \right)f_{2}$&$f_{1}e_{1}=\left( 1-\alpha \right) f_{2}$&$f_{1}f_{2}=\allowbreak 2\frac{\alpha -1}{\alpha +1}e_{1}$&$f_{2}f_{1}=-\frac{4}{\alpha +1}e_{1}$\\

${\rm R}_{02}$&$:$&$e_{1}f_{1}=4f_{2}$&$f_{1}e_{1}=-2f_{2}$&$f_{1}f_{1}=e_{1}$\\&&$f_{1}f_{2}=e_{1}$&$f_{2}f_{1}=-e_{1}$\\

${\rm R}_{03}$&$:$&$e_{1}f_{1}=f_{2}$&$f_{1}e_{1}=-f_{2}$&$f_{1}f_{2}=2e_{1}$\\

${\rm R}_{04}^{\alpha }$&$:$&$ f_{1}f_{2}=\left( \alpha +1\right)e_{1}$&$f_{2}f_{1}=\left( \alpha -1\right) e_{1}$\\

${\rm R}_{05}$&$:$&$ f_{1}f_{2}=e_{1}$&$f_{2}f_{1}=-e_{1}$&$f_{2}f_{2}=-e_{1}$\\

${\rm R}_{06}^{\alpha }$&$:$&$ e_{1}f_{1}=\left( 1+\alpha \right)f_{2}$&$f_{1}e_{1}=\left( 1-\alpha \right) f_{2}$\\

${\rm R}_{07}^{\alpha }$&$:$&$ e_{1}f_{1}=\left( 1+\alpha \right)f_{2}$&$f_{1}e_{1}=\left( 1-\alpha \right) f_{2}$&$f_{1}f_{1}=e_{1}$\\

${\rm R}_{08}$&$:$&$ f_{1}e_{1}=2f_{2}$&$f_{1}f_{2}=e_{1}$&$f_{2}f_{1}=e_{1}$\\

${\rm R}_{09}$&$:$&$ e_{1}e_{1}=e_{1}$&$f_{1}e_{1}=f_{1}$\\

${\rm R}_{10}$&$:$&$ e_{1}e_{1}=e_{1}$&$e_{1}f_{1}=f_{2}$&$f_{1}e_{1}=f_{1}-f_{2}$\\

${\rm R}_{11}$&$:$&$ e_{1}e_{1}=e_{1}$&$e_{1}f_{1}=f_{1}$\\

${\rm R}_{12}$&$:$&$ e_{1}e_{1}=e_{1}$&$e_{1}f_{1}=f_{1}$&$f_{1}f_{1}=e_{1}$\\

${\rm R}_{13}$&$:$&$e_{1}e_{1}=e_{1}$&$e_{1}f_{1}=f_{1}$&$e_{1}f_{2}=f_{2}$&$f_{1}e_{1}=f_{1}$\\

${\rm R}_{14}$&$:$&$e_{1}e_{1}=e_{1}$&$e_{1}f_{1}=f_{1}$&$e_{1}f_{2}=f_{1}+f_{2}$\\&
&$f_{1}e_{1}=f_{1}$&$f_{2}e_{1}=-f_{1}$\\

${\rm R}_{15}$&$:$&$e_{1}e_{1}=e_{1}$&$e_{1}f_{1}=f_{1}$&$f_{1}e_{1}=f_{1}$&$f_{2}e_{1}=f_{2}$\\

${\rm R}_{16}$&$:$&$e_{1}e_{1}=e_{1}$&$e_{1}f_{1}=f_{1}$&$f_{1}e_{1}=f_{1}$\\
&&$f_{1}f_{1}=e_{1}$&$f_{2}e_{1}=f_{2}$\\

${\rm R}_{17}$&$:$&$e_{1}e_{1}=e_{1}$&$e_{1}f_{1}=f_{1}$&$f_{1}e_{1}=f_{1}$\\

${\rm R}_{18}$&$:$&$e_{1}e_{1}=e_{1}$&$e_{1}f_{1}=f_{1}$&$f_{1}e_{1}=f_{1}$&$f_{1}f_{1}=e_{1}$\\

${\rm R}_{19}$&$:$&$e_{1}e_{1}=e_{1}$\\

${\rm R}_{20}$&$:$&$e_{1}e_{1}=e_{1}$&$e_{1}f_{1}=f_{1}$&$f_{2}e_{1}=f_{2}$\\

${\rm R}_{21}$&$:$&$e_{1}e_{1}=e_{1}$&$e_{1}f_{1}=f_{1}$&$f_{2}e_{1}=f_{2}$&$f_{1}f_{1}=e_{1}$\\

${\rm R}_{22}$&$:$&$e_{1}e_{1}=e_{1}$&$f_{1}e_{1}=f_{1}$&$f_{2}e_{1}=f_{2}$\\

${\rm R}_{23}$&$:$&$e_{1}e_{1}=e_{1}$&$e_{1}f_{1}=f_{1}$&$e_{1}f_{2}=f_{2}$\\

${\rm R}_{24}$&$:$&
$e_{1}e_{1}=e_{1}$&$e_{1}f_{1}=f_{1}$&$f_{1}e_{1}=f_{1}$&$e_{1}f_{2}=f_{2}$&$f_{2}e_{1}=f_{2}$\\

${\rm R}_{25}$&$:$&$f_{1}f_{2}=e_{1}$&$f_{2}f_{1}=e_{1}$\\

${\rm R}_{26}$&$:$&$f_{1}f_{1}=e_{1}$\\

${\rm R}_{27}$&$:$&$e_{1}f_{1}=f_{2}$&$f_{1}e_{1}=-f_{2}$\\

${\rm R}_{28}$&$:$&$e_{1}f_{1}=f_{2}$&$f_{1}e_{1}=-f_{2}$&$f_{1}f_{1}=e_{1}$
\end{longtable}
\noindent{}All listed superalgebras are non-isomorphic except: ${\rm R}_{04}^{\alpha
}\cong {\rm R}_{04}^{-\alpha }$.
\end{theoremA1}

\subsubsection{The proof of  Theorem A1}
Let ${\rm R}$ be a nontrivial $3$-dimensional right alternative
superalgebra of type $\left( 1,2\right) $. Then ${\rm R}^{+}$ is a $3$%
-dimensional Jordan superalgebra of type $\left( 1,2\right) $. Then we may
assume ${\rm R}^{+}\in \left\{ {\rm J}_{01},\ldots ,{\rm J}%
_{12}\right\} $. So we have the following cases:

\begin{enumerate}[I.]
    \item 
\underline{${\rm R}^{+}={\rm J}_{01}$}. Let $\theta =\big({\rm B}_{1},\ {\rm B}_{2},\ {\rm B}_{3}\big) $ be an arbitrary element of ${\rm Z}^{2}\left( \rm{J}_{01},{\rm J}_{01}\right) $. Then \begin{center}
$\theta\ =\ \left( \alpha _{1}\Delta _{22}+
\frac{\alpha _{2}-3}{\alpha _{2}+1}\Delta _{23},\ 0,\ \alpha _{2}\Delta
_{12}\right),$
\end{center} for some $\alpha _{1},\alpha _{2}\in \mathbb{C}$ with $\alpha
_{2}\neq -1$. The automorphism group  ${\rm Aut}\left( 
{\rm J}_{01}\right) $ consists of the automorphisms $\phi $ given by a
matrix of the following form:
\begin{equation*}
\phi =%
\begin{pmatrix}
a_{11} & 0 & 0 \\ 
0 & \epsilon & 0 \\ 
0 & a_{32} & \epsilon a_{11}%
\end{pmatrix}%
:\epsilon ^{2}=1.
\end{equation*}

\begin{itemize}
\item $\alpha _{2}\neq 3$. We choose to be the following automorphism:%
\begin{equation*}
\phi =%
\begin{pmatrix}
1 & 0 & 0 \\ 
0 & 1 & 0 \\ 
0 & \frac{\alpha _{1}\left( 1+\alpha _{2}\right) }{6-2\alpha _{2}} & 1%
\end{pmatrix}%
.
\end{equation*}%
Then $\theta \ast \phi \ =\ \left( \frac{\alpha _{2}-3}{\alpha _{2}+1}\Delta _{23},\ 0,\ \alpha
_{2}\Delta _{12}\right) $. So we get the superalgebras ${\rm R}_{01}^{\alpha
\notin \{ -1,3\}}$.

\item $\alpha _{2}=3$. If $\alpha _{1}=0$, we have 
$\theta =\big(0, \ 0,\ 3\Delta _{12}\big) $ and we obtain the superalgebra ${\rm R}_{01}^{3}$. If $\alpha _{1}\neq 0$, we choose $\phi $\ to be the
following automorphism:%
\begin{equation*}
\phi =%
\begin{pmatrix}
\alpha _{1} & 0 & 0 \\ 
0 & 1 & 0 \\ 
0 & 0 & \alpha _{1}%
\end{pmatrix}%
.
\end{equation*}%
Then $\theta \ast \phi =\big( \Delta _{22},\ 0,\ 3\Delta _{12}\big) $. Hence
we get the superalgebra ${\rm R}_{02}$.
\end{itemize}

    \item 
\underline{${\rm R}^{+}={\rm J}_{02}$}. Let $\theta =\big(
{\rm B}_{1},\ {\rm B}_{2},\ {\rm B}_{3}\big) $ be an arbitrary element of 
${\rm Z}^{2}\left( \rm{J}_{02},{\rm J}_{02}\right) $. Then $\theta \in \left\{ \eta _{1},\ldots
,\eta _{4}\right\} $ where%
\begin{longtable}{lcl}
$\eta _{1} $&$=$&$\big( \frac{\alpha _{2}\left( 1-\alpha _{3}\right) }{\alpha
_{1}}\Delta _{22}+\alpha _{3}\Delta _{23}-\frac{\alpha _{1}\left( 1+\alpha
_{3}\right) }{\alpha _{2}}\Delta _{33},\ 
\alpha _{1}\Delta _{12}-\frac{\alpha
_{1}^{2}}{\alpha _{2}}\Delta _{13},\ \alpha _{2}\Delta _{12}-\alpha _{1}\Delta
_{13}\big) ,$ \\
$\eta _{2} $&$=$&$\big( \alpha _{1}\Delta _{22}+\alpha _{2}\Delta _{33}+\alpha
_{3}\Delta _{23},\ 0,\ 0\big),$ \\
$\eta _{3} $&$=$&$\big( \alpha _{1}\Delta _{33}-\Delta _{23},\ \alpha _{2}\Delta
_{13},\ 0\big) ,$ \\
$\eta _{4} $&$=$&$\big( \alpha _{1}\Delta _{22}+\Delta _{23},\ 0,\ \alpha _{2}\Delta
_{12}\big) ,$
\end{longtable}%
for some $\alpha _{1},\alpha _{2},\alpha _{3}\in \mathbb{C}$. The
automorphism group ${\rm Aut}\left( {\rm J}_{02}\right) $ consists of the automorphisms $\phi $ given by a matrix of the following form:
\begin{equation*}
\phi =%
\begin{pmatrix}
a_{22}a_{33}-a_{23}a_{32} & 0 & 0 \\ 
0 & a_{22} & a_{23} \\ 
0 & a_{32} & a_{33}%
\end{pmatrix}%
.
\end{equation*}

\begin{itemize}
\item $\theta =\eta _{1}$. Let $\phi $ be the following automorphism:%
\begin{equation*}
\begin{pmatrix}
\frac{1}{\sqrt{\alpha _{2}}} & 0 & 0 \\ 
0 & \frac{\alpha _{3}+1}{2\sqrt{\alpha _{2}}} & \frac{\alpha
_{1}}{\alpha _{2}} \\ 
0 & \frac{ \sqrt{\alpha _{2}} (\alpha _{3}-1)}{2\alpha _{1}} & 1%
\end{pmatrix}%
.
\end{equation*}%
Then $\theta \ast \phi =\big( \Delta _{23},\ 0,\ \Delta _{12}\big) $. Thus we
get the superalgebra ${\rm R}_{03}$.

\item $\theta =\eta _{2}$. Let $\phi =\bigl(a_{ij}\bigr)\in $ ${\rm Aut}%
\left( {\rm J}_{02}\right) $. Then $\theta \ast \phi\ =\ \left( \beta
_{1}\Delta _{22}+\beta _{2}\Delta _{33}+\beta _{3}\Delta _{23},\ 0,\ 0\right),$
where%
\begin{longtable}{lcl}
$\beta _{1} $&$=$&$\frac{\alpha
_{1}a_{22}^{2}+2\alpha _{3}a_{22}a_{32}+\alpha _{2}a_{32}^{2}}{a_{22}a_{33}-a_{23}a_{32}} ,$ \\
$\beta _{2} $&$=$&$\frac{\alpha
_{1}a_{23}^{2}+2\alpha _{3}a_{23}a_{33}+\alpha _{2}a_{33}^{2}}{a_{22}a_{33}-a_{23}a_{32}} ,$ \\
$\beta _{3} $&$=$&$\frac{\alpha
_{1}a_{22}a_{23}+\alpha _{2}a_{32}a_{33}+\alpha _{3}a_{22}a_{33}+\alpha
_{3}a_{23}a_{32}}{a_{22}a_{33}-a_{23}a_{32}}.$
\end{longtable}%
Whence%
\begin{equation*}
\begin{pmatrix}
-\beta _{3} & -\beta _{2} \\ 
\beta _{1} & \beta _{3}%
\end{pmatrix} \ 
= \ 
\begin{pmatrix}
a_{22} & a_{23} \\ 
a_{32} & a_{33}%
\end{pmatrix}%
^{-1}%
\begin{pmatrix}
-\alpha _{3} & -\alpha _{2} \\ 
\alpha _{1} & \alpha _{3}%
\end{pmatrix}%
\begin{pmatrix}
a_{22} & a_{23} \\ 
a_{32} & a_{33}%
\end{pmatrix}%
.
\end{equation*}%
So we may assume $%
\begin{pmatrix}
-\alpha _{3} & -\alpha _{2} \\ 
\alpha _{1} & \alpha _{3}%
\end{pmatrix}%
\in \left\{ 
\begin{pmatrix}
-\alpha & 0 \\ 
0 & \alpha%
\end{pmatrix}%
,%
\begin{pmatrix}
0 & 1 \\ 
0 & 0%
\end{pmatrix}%
\right\} $. Therefore we obtain the superalgebras ${\rm R}_{04}^{\alpha }$ and 
${\rm R}_{05}$.

\item $\theta =\eta _{3}$. Let $\phi $\ be the following automorphism:%
\begin{equation*}
\phi =%
\begin{pmatrix}
\frac{\bf i}{\sqrt{\alpha _{2}}} & 0 & 0 \\ 
0 & \frac{\alpha _{1}{\bf i}}{2\sqrt{\alpha _{2}}} & -1 \\ 
0 & \frac{\bf  i}{\sqrt{\alpha _{2}}} & 0%
\end{pmatrix}%
.
\end{equation*}%
Then $\theta \ast \phi =\big( \Delta _{23},\ 0,\ \Delta _{12}\big) $. So we
have the superalgebra ${\rm R}_{03}$.

\item $\theta =\eta _{4}$. Let $\phi $\ be the following automorphism:%
\begin{equation*}
\phi =%
\begin{pmatrix}
\frac{1}{\sqrt{\alpha _{2}}} & 0 & 0 \\ 
0 & \frac{1}{\sqrt{\alpha _{2}}} & 0 \\ 
0 & - \frac{\alpha _{1}}{2\sqrt{\alpha _{2}}} & 1%
\end{pmatrix}%
.
\end{equation*}%
Then $\theta \ast \phi =\big( \Delta _{23},\ 0,\ \Delta _{12}\big) $. Hence
we have the superalgebra ${\rm R}_{03}$.
\end{itemize}

    \item 
\underline{${\rm R}^{+}={\rm J}_{03}$}. Let $\theta =\big({\rm B}_{1},\ {\rm B}_{2},\ {\rm B}_{3}\big) $ be an arbitrary element of ${\rm Z}^{2}\left( \rm{J}_{03},{\rm J}_{03}\right) $. Then $\theta \in \left\{ \eta _{1},\eta
_{2}\right\} $ where%
\begin{longtable}{lcl}
$\eta _{1} $&$=$&$\left( \alpha _{1}\Delta _{22}+\alpha _{2}\Delta
_{23},\ 0,\ -\Delta _{12}\right),$ \\
$\eta _{2} $&$=$&$\left( \alpha _{1}\Delta _{22},\ 0,\ \alpha _{2}\Delta _{12}\right)_{\alpha _{2}\neq -1},$
\end{longtable}%
for some $\alpha _{1},\alpha _{2}\in \mathbb{C}$. The automorphism group   ${\rm Aut}\left( {\rm J}_{03}\right) $ consists of the automorphisms $\phi $ given by a matrix of the following form:
\begin{equation*}
\phi =
\begin{pmatrix}
a_{11} & 0 & 0 \\ 
0 & a_{22} & 0 \\ 
0 & a_{32} & a_{11}a_{22}%
\end{pmatrix}
.
\end{equation*}

\begin{itemize}
\item $\theta =\eta _{1}$. Let $\phi =\bigl(a_{ij}\bigr)\in $ ${\rm Aut}%
\left( {\rm J}_{03}\right) $. Then $\theta \ast \phi \ =\ \left( \beta
_{1}\Delta _{22}+\beta _{2}\Delta _{23},\ 0,\ -\Delta _{12}\right),$ where 
\begin{longtable}{lcllcl}
$
\beta _{1}$&$=$&${a_{11}^{-1}}a_{22}\left( \alpha _{1}a_{22}+2\alpha
_{2}a_{32}\right),$ & $\beta _{2}$&$=$&$\alpha _{2}a_{22}^{2}.$
\end{longtable}So we have the
representatives 
\begin{center}
    $\big( 0,\ 0,\ -\Delta _{12}\big),$ \ 
    $\big( \Delta_{22},\ 0,\ -\Delta _{12}\big),$ \ 
    $\big( \Delta _{23},\ 0,\ -\Delta _{12}\big) $.
\end{center}So we get the superalgebras ${\rm R}_{06}^{-1},$\ ${\rm R}_{07}^{-1},$ and ${\rm R}_{08}$.

\item $\theta =\eta _{2}$. Let $\phi =\bigl(a_{ij}\bigr)\in $ ${\rm Aut}%
\left( {\rm J}_{03}\right) $. Then $\theta \ast \phi =\left( \beta
_{1}\Delta _{22},\ 0,\ \beta _{2}\Delta _{12}\right),$ 
where $\beta _{1}\ =\ 
\alpha _{1}{a_{11}^{-1}}a_{22}^{2}$ and $\beta _{2}\ =\ \alpha _{2}$. 
So we have the
representatives $
\big( 0,\ 0 ,\ \alpha \Delta _{12}\big)$ and  
$\big( \Delta _{22},\ 0,\ \alpha \Delta _{12}\big)$ for $\alpha \neq -1.$ 
So we get the
superalgebras ${\rm R}_{06}^{\alpha \neq -1}$ and ${\rm R}_{07}^{\alpha \neq -1}$.
\end{itemize}

    \item 
\underline{${\rm R}^{+}={\rm J}_{04}$}. Let $\theta =\big(
{\rm B}_{1},\ {\rm B}_{2},\ {\rm B}_{3}\big) $ be an arbitrary element of ${\rm Z}^{2}\left( \rm{J
}_{4},{\rm J}_{04}\right) $. Then $\theta \in \left\{ \eta _{1},\eta
_{2}\right\} $ where%
\begin{longtable}{lcllcl}
$\eta _{1} $&$=$&$\left( 0,\ -\frac{1}{2}\Delta _{12},\ \alpha _{1}\Delta
_{12}\right),$ & 
$\eta _{2} $&$=$&$\left( \alpha _{1}\Delta _{22},\ \frac{1}{2}\Delta _{12},\ 0\right),$
\end{longtable}
for some $\alpha _{1}\in \mathbb{C}$. The automorphism group  ${\rm Aut}\left( {\rm J}_{04}\right) $  consists of the automorphisms $\phi$ given by a matrix of the following form:
\begin{equation*}
\phi =
\begin{pmatrix}
1 & 0 & 0 \\ 
0 & a_{22} & 0 \\ 
0 & 0 & a_{33}%
\end{pmatrix}%
.
\end{equation*}

\begin{itemize}
\item $\theta =\eta _{1}$. Let $\phi =\bigl(a_{ij}\bigr)\in $ ${\rm Aut}%
\left( {\rm J}_{04}\right) $. Then $\theta \ast \phi \ =\ \left( 0,\ -\frac{1}{2%
}\Delta _{12},\ \alpha _{1}{a_{22}}{a_{33}^{-1}}\Delta _{12}\right) $. Hence
we have representatives 
$\left( 0,\ -\frac{1}{2}\Delta _{12},\ 0\right) $
and $\left( 0,\ -\frac{1}{2}\Delta _{12},\ \Delta _{12}\right) $. Therefore we
obtain the superalgebras ${\rm R}_{09}$ and ${\rm R}_{10}$.

\item $\theta =\eta _{2}$. Let $\phi =\bigl(a_{ij}\bigr)\in $ ${\rm Aut}%
\left( {\rm J}_{04}\right) $. Then $\theta \ast \phi \ =\ \left( \alpha
_{1}a_{22}^{2}\Delta _{22},\ \frac{1}{2}\Delta _{12},\ 0\right) $. Hence we have
the representatives $\left( 0,\ \frac{1}{2}\Delta _{12},\ 0\right) $ and 
$\left(\Delta _{22},\ \frac{1}{2}\Delta _{12},\ 0\right) $. 
So we get the superalgebras $%
{\rm R}_{11}$ and ${\rm R}_{12}$.
\end{itemize}

    \item 
\underline{${\rm R}^{+}={\rm J}_{05}$}. Let $\theta =\big(
{\rm B}_{1},\ {\rm B}_{2},\ {\rm B}_{3}\big) $ be an arbitrary element of ${\rm Z}^{2}\left( \rm{J
}_{05},{\rm J}_{05}\right) $. Then $\theta \in \left\{ \eta _{1},\eta
_{2}\right\},$ where%
\begin{longtable}{lcllcl}
$\eta _{1} $&$=$&$\left( 0,\ \alpha _{1}\Delta _{13},\ \frac{1}{2}\Delta _{13}\right),$ &
$\eta _{2} $&$=$&$\left( \alpha _{1}\Delta _{22},\ 0,\ -\frac{1}{2}\Delta_{13}\right),$
\end{longtable}
for some $\alpha _{1}\in \mathbb{C}$. The automorphism group  ${\rm Aut}\left( {\rm J}_{05}\right) $ consists of the automorphisms $\phi $ given by a matrix of the following form:
\begin{equation*}
\phi =%
\begin{pmatrix}
1 & 0 & 0 \\ 
0 & a_{22} & 0 \\ 
0 & 0 & a_{33}%
\end{pmatrix}%
.
\end{equation*}

\begin{itemize}
\item $\theta =\eta _{1}$. Let $\phi =\bigl(a_{ij}\bigr)\in $ ${\rm Aut}%
\left( {\rm J}_{05}\right) $. Then 
$\theta \ast \phi \ =\ \left( 0,\ {\alpha _{1}}{a_{22}^{-1}}a_{33}\Delta _{13},\ \frac{1}{2}\Delta _{13}\right).$
Hence we have representatives 
$\left( 0,\ 0,\ \frac{1}{2}\Delta _{13}\right)$ and $\left( 0,\ \Delta _{13},\ \frac{1}{2}\Delta _{13}\right) $. Thus we get the
algebras ${\rm R}_{13}$ and ${\rm R}_{14}$.

\item $\theta =\eta _{2}$. Let $\phi =\bigl(a_{ij}\bigr)\in $ ${\rm Aut}%
\left( {\rm J}_{05}\right) $. Then 
$\theta \ast \phi \ =\ \left( \alpha_{1}a_{22}^{2}\Delta _{22},\ 0,\ -\frac{1}{2}\Delta _{13}\right) $. Hence we
have the representatives $\left( 0,\ 0,\ -\frac{1}{2}\Delta _{13}\right)$ and  
$\left(\Delta _{22},\ 0,\ -\frac{1}{2}\Delta _{13}\right) $. So we get the superalgebras ${\rm R}_{15}$ and ${\rm R}_{16}$.
\end{itemize}

    \item 
\underline{${\rm R}^{+}={\rm J}_{06}$}. Let $\theta =\big(
{\rm B}_{1},\ {\rm B}_{2},\ {\rm B}_{3}\big) $ be an arbitrary element of ${\rm Z}^{2}\left( \rm{J}_{06},{\rm J}_{06}\right) $. Then $\theta \ =\ \big( \alpha _{1}\Delta
_{22},\ 0,\ 0\big) $ for some $\alpha _{1}\in \mathbb{C}$. The automorphism
group ${\rm Aut}\left( {\rm J}_{06}\right) $
consists of the automorphisms $\phi $ given by a matrix of the following
form:%
\begin{equation*}
\phi =%
\begin{pmatrix}
1 & 0 & 0 \\ 
0 & a_{22} & 0 \\ 
0 & 0 & a_{33}%
\end{pmatrix}%
.
\end{equation*}%
Let $\phi =\bigl(a_{ij}\bigr)\in $ ${\rm Aut}\left( {\rm J}_{06}\right).$ 
Then $\theta \ast \phi \ =\ \left( \alpha _{1}a_{22}^{2}\Delta_{22},\ 0,\ 0\right) $. Hence we have the representative $\big( 0,\ 0,\ 0\big)$ and $\big( \Delta _{22},\ 0,\ 0\big)$. Thus we get the superalgebras ${\rm R}
_{17} $ and ${\rm R}_{18}$.

    \item 
\underline{${\rm R}^{+}={\rm J}_{07}$}. Then ${\rm Z}^{2}\left(
{\rm J}_{07},{\rm J}_{07}\right) =\left\{ 0\right\} $. So we get the superalgebra ${\rm R}
_{19}$.

    \item 
\underline{${\rm R}^{+}={\rm J}_{08}$}. Let $\theta =\left(
{\rm B}_{1},\ {\rm B}_{2},\ {\rm B}_{3}\right) $ be an arbitrary element of ${\rm Z}^{2}\left( \rm{J
}_{08},{\rm J}_{08}\right) $.Then $\theta \in \left\{ \eta _{1},\ldots
,\eta _{7}\right\} $ where%
\begin{longtable}{lcl}
$\eta _{1} $&$=$&$\left( \frac{2\alpha _{1}\alpha _{3}}{1-2\alpha _{2}}\Delta
_{22}+\alpha _{1}\Delta _{23}+\frac{\alpha _{1}\left( 1-2\alpha _{2}\right) 
}{2\alpha _{3}}\Delta _{33},\ \alpha _{2}\Delta _{12}+\frac{1-4\alpha _{2}^{2}%
}{4\alpha _{3}}\Delta _{13},\ \alpha _{3}\Delta _{12}-\alpha _{2}\Delta
_{13}\right) ,$ \\
$\eta _{2} $&$=$&$\left( \frac{\alpha _{1}}{\alpha _{2}}\Delta _{22}+\alpha
_{1}\Delta _{23}+\alpha _{1}\alpha _{2}\Delta _{33},\ 
\frac{1}{2}\Delta_{12}+\alpha _{2}\Delta _{13},\ 
-\frac{1}{2}\Delta _{13}\right) ,$ \\
$\eta _{3} $&$=$&$\left( \alpha _{1}\Delta _{33},\ 
-\frac{1}{2}\Delta _{12}+\alpha_{2}\Delta _{13},\ 
\frac{1}{2}\Delta _{13}\right) ,$ \\
$\eta _{4} $&$=$&$\left( \alpha _{1}\Delta _{22},\ 
\frac{1}{2}\Delta _{12},\ 
\alpha_{2}\Delta _{12}-\frac{1}{2}\Delta _{13}\right) ,$ \\
$\eta _{5} $&$=$&$\left(0,\ 
-\frac{1}{2}\Delta _{12},\ 
-\frac{1}{2}\Delta_{13}\right) ,$ \\
$\eta _{6} $&$=$&$\left(0,\ 
\frac{1}{2}\Delta _{12},\ 
\frac{1}{2}\Delta _{13}\right),$
\end{longtable}
for some $\alpha _{1},\alpha _{2},\alpha _{3}\in \mathbb{C}$. The
automorphism group  ${\rm Aut}\left( {\rm J}
_{08}\right) $ consists of the automorphisms $\phi $ given by a matrix of
the following form: 
\begin{equation*}
\phi =%
\begin{pmatrix}
1 & 0 & 0 \\ 
0 & a_{22} & a_{23} \\ 
0 & a_{32} & a_{33}%
\end{pmatrix}%
.
\end{equation*}

\begin{itemize}
\item $\theta =\eta _{1}$. Let $\phi $ be the first of the following
matrices if $\alpha _{1}=0$ or the second if $\alpha _{1}\neq 0$:%
\begin{equation*}
\begin{pmatrix}
1 & 0 & 0 \\ 
0 & \alpha _{2}+\frac{1}{2} & \frac{2\alpha
_{2}-1}{2\alpha _{3}} \\ 
0 & \alpha _{3} & 1%
\end{pmatrix}%
,\  
\begin{pmatrix}
1 & 0 & 0 \\ 
0 & \frac{2\alpha_2+1}{2}\sqrt{\frac{1-2\alpha _{2}}{2\alpha_1\alpha _{3}}} & 
\frac{2\alpha _{2}-1}{2\alpha _{3}} \\ 
0 & \sqrt{\frac{\alpha _{3}(1-2\alpha _{2})}{2\alpha_1}} & 1
\end{pmatrix}
\end{equation*}
Then $\theta \ast \phi \in \left\{ \left( 0,\ \frac{1}{2}\Delta _{12},\ -\frac{1}{2}\Delta _{13}\right) ,\ 
\left( \Delta _{22},\frac{1}{2}\Delta _{12},-\frac{1%
}{2}\Delta _{13}\right) \right\} $. So we get the superalgebras ${\rm R}_{20}$%
\ and ${\rm R}_{21}$.


\item $\theta =\eta _{2}$. Let $\phi=\phi_1$  if $\alpha _{1}=0$ and 
$\phi=\phi_2$  if  $\alpha _{1}\neq 0$:%
\begin{equation*}
\phi_1 \ = \ \begin{pmatrix}
1 & 0 & 0 \\ 
0 & 1 & -\alpha _{2} \\ 
0 & 0 & 1%
\end{pmatrix}%
, \ 
\phi_2 \ = \ \begin{pmatrix}
1 & 0 & 0 \\ 
0 & \sqrt{\frac{\alpha _{2}}{\alpha _{1}}} & 
-\alpha _{2} \\ 
0 & 0 & 1%
\end{pmatrix}%
.
\end{equation*}%
Then $\theta \ast \phi \in 
\left\{ \left( 0,\ \frac{1}{2}\Delta _{12},\ -\frac{1}{2}\Delta _{13}\right),\ 
\left( \Delta _{22},\ \frac{1}{2}\Delta _{12},\ -\frac{1}{2}\Delta _{13}\right) \right\} $. Thus we have the superalgebras ${\rm R}_{20}$ and ${\rm R}_{21}$.

\item $\theta =\eta _{3}$. 
Let $\phi=\phi_1$   if $\alpha _{1}=0$ and $\phi=\phi_2$  if $\alpha _{1}\neq 0$:
\begin{equation*}
\phi_1\ = \ \begin{pmatrix}
1 & 0 & 0 \\ 
0 & \alpha _{2} & 1 \\ 
0 & 1 & 0%
\end{pmatrix},%
\ 
\phi_2 \ = \ \begin{pmatrix}
1 & 0 & 0 \\ 
0 & \frac{1}{\sqrt{\alpha _{1}}}\alpha _{2} & 1 \\ 
0 & \frac{1}{\sqrt{\alpha _{1}}} & 0%
\end{pmatrix}%
.
\end{equation*}%

Then $\theta \ =\ \left( 0,\ \frac{1}{2}\Delta _{12},\ -\frac{1}{2}\Delta
_{13}\right)$ and $\theta \ast \phi \ =\ \left( \Delta _{22},\ \frac{1}{2}\Delta _{12},\ -\frac{1}{2}\Delta _{13}\right) $. Hence, in this case we obtain the superalgebra ${\rm R}_{20}$ and ${\rm R}_{21}$.

\item $\theta =\eta _{4}$. Let $\phi=\phi_1 $  if $\alpha _{1}=0$ and  
$\phi=\phi_2$ if $\alpha _{1}\neq 0$:%
\begin{equation*}
\phi_1 \ = \ \begin{pmatrix}
1 & 0 & 0 \\ 
0 & 1 & 0 \\ 
0 & \alpha _{2} & 1%
\end{pmatrix}, \ 
\phi_2 \ = \ \begin{pmatrix}
1 & 0 & 0 \\ 
0 &  {\alpha^{-\frac 12}_{1}} & 0 \\ 
0 & {\alpha _{2}} {\alpha^{-\frac 12}_{1}} & 1%
\end{pmatrix}%
\end{equation*}%
Then $\theta \ast \phi \in \left\{ 
\left( 0,\ \frac{1}{2}\Delta _{12},\ -\frac{1}{2}\Delta _{13}\right), \ 
\left( \Delta _{22},\ \frac{1}{2}\Delta _{12},\ -\frac{1}{2}\Delta _{13}\right) \right\} $. So we have the superalgebras ${\rm R}_{20} $ and ${\rm R}_{21}$.

\item $\theta =\eta _{5}$. We obtain the superalgebra ${\rm R}_{22}$.

\item $\theta =\eta _{6}$. We obtain the superalgebra ${\rm R}_{23}$.
\end{itemize}

    \item 
\underline{${\rm R}^{+}={\rm J}_{09}$}. Then ${\rm Z}^{2}\left( {\rm J}%
_{09},{\rm J}_{09}\right) =\left\{ 0\right\} $. So we get the superalgebra $%
{\rm R}_{24}$.

    \item 
\underline{${\rm R}^{+}={\rm J}_{10}$}. Then ${\rm Z}^{2}\left( {\rm J}%
_{10},{\rm J}_{10}\right) =\emptyset $. Thus there is no right
alternative superalgebras with ${\rm R}^{+}={\rm J}_{10}$.

    \item 
\underline{${\rm R}^{+}={\rm J}_{11}$}. Then ${\rm Z}^{2}\left( {\rm J}%
_{11},{\rm J}_{11}\right) =\emptyset $. Thus there is no right
alternative superalgebras with ${\rm R}^{+}={\rm J}_{11}$.

    \item 
\underline{${\rm R}^{+}={\rm J}_{12}$}. Let $\theta =
\big({\rm B}_{1},\ {\rm B}_{2},\ {\rm B}_{3}\big) \neq 0$ be an arbitrary element of ${\rm Z}^{2}\left( 
{\rm J}_{12},{\rm J}_{12}\right) $. Then $\theta \in \left\{ \eta
_{1},\ldots ,\eta _{4}\right\} $ where%
\begin{longtable}{lcl}
$\eta _{1} $&$=$&$\big( \alpha _{1}\Delta _{22}+\alpha _{2}\Delta _{33}+\alpha
_{3}\Delta _{23},\ 0,\ 0\big) ,$ \\
$\eta _{2} $&$=$&$\big( -\frac{\alpha _{1}\alpha _{3}}{\alpha _{2}}\Delta
_{22}+\alpha _{1}\Delta _{23}-\frac{\alpha _{1}\alpha _{2}}{\alpha _{3}}%
\Delta _{33},\ \alpha _{2}\Delta _{12}-\frac{\alpha _{2}^{2}}{\alpha _{3}}%
\Delta _{13}, \ \alpha _{3}\Delta _{12}-\alpha _{2}\Delta _{13}\big) _{\alpha
_{2}\alpha _{3}\neq 0},$ \\
$\eta _{3} $&$=$&$\big( \alpha _{1}\Delta _{33},\ 
\alpha _{2}\Delta _{13},\ 0 \big)_{\alpha _{2}\neq 0},$ \\
$\eta _{4} $&$=$&$\big( \alpha _{1}\Delta _{22},\ 0,\ \alpha _{2}\Delta _{12}\big)_{\alpha _{2}\neq 0},$
\end{longtable}
for some $\alpha _{1},\alpha _{2},\alpha _{3}\in \mathbb{C}$. The
automorphism group   ${\rm Aut}\left( {\rm J}_{12}\right) $, consists of the automorphisms $\phi $ given by a matrix of
the following form:%
\begin{equation*}
\phi =%
\begin{pmatrix}
a_{11} & 0 & 0 \\ 
0 & a_{22} & a_{23} \\ 
0 & a_{32} & a_{33}%
\end{pmatrix}%
.
\end{equation*}

\begin{itemize}
\item $\theta =\eta _{1}$. Let $\phi =\bigl(a_{ij}\bigr)\in $ ${\rm Aut}%
\left( {\rm J}_{12}\right) $. Then 
$\theta \ast \phi \ =\ \left( \beta_{1}\Delta _{22}+\beta _{2}\Delta _{33}+\beta _{3}\Delta _{23},\ 0,\ 0\right),$
where
\begin{longtable}{lcl}
$\beta _{1} $&$=$&${a^{-1}_{11}}\big( \alpha _{1}a_{22}^{2}+2\alpha
_{3}a_{22}a_{32}+\alpha _{2}a_{32}^{2}\big) ,$ \\
$\beta _{2} $&$=$&${a^{-1}_{11}}\big( \alpha _{1}a_{23}^{2}+2\alpha
_{3}a_{23}a_{33}+\alpha _{2}a_{33}^{2}\big) ,$ \\
$\beta _{3} $&$=$&${a^{-1}_{11}}\big( \alpha _{1}a_{22}a_{23}+\alpha
_{2}a_{32}a_{33}+\alpha _{3}a_{22}a_{33}+\alpha _{3}a_{23}a_{32}\big).$
\end{longtable}%
So we have the representatives 
$\big( \Delta _{23},\ 0,\ 0\big) $ and $\big(\Delta _{22},\ 0,\ 0\big)$. 
Therefore we get the superalgebras ${\rm R}_{25}$
and ${\rm R}_{26}$.

\item $\theta =\eta _{2}$. Let 
$\phi=\phi_1 $  if $\alpha _{1}=0$ and 
$\phi=\phi_2$  if $\alpha _{1}\neq 0$:%
\begin{equation*}
\phi_1\ =\ \begin{pmatrix}
1 & 0 & 0 \\ 
0 & 0 & 1 \\ 
0 & -{\alpha _{3}}{\alpha _{2}^{-2}} & {\alpha _{3}}{\alpha^{-1}_{2}}%
\end{pmatrix}%
, \ 
\phi_2 \ = \ \begin{pmatrix}
-{\alpha _{1}}{\alpha _{2}^{2}}\alpha _{3}^{-1} & 0 & 0 \\ 
0 & 0 & {\alpha _{1}}{\alpha _{2}^{3}}\alpha _{3}^{-2} \\ 
0 & 1 & {\alpha _{1}}{\alpha_{2}^{2}}\alpha _{3}^{-1}%
\end{pmatrix}.%
\end{equation*}%
Then $\theta \ast \phi \in \big\{ \big( 0,\ 0,\ \Delta _{12}\big) ,\ 
\big(\Delta _{22},\ 0,\ \Delta _{12}\big) \big\} $. So we get  ${\rm R}_{27}$ and ${\rm R}_{28}$.

\item $\theta =\eta _{3}$. Let 
$\phi=\phi_1 $  if $\alpha _{1}=0$ and  $\phi=\phi_2$  if $\alpha _{1}\neq 0$:%
\begin{equation*}
\phi_1 \ = \ \begin{pmatrix}
1 & 0 & 0 \\ 
0 & 1 & \alpha _{2} \\ 
0 & 1 & 0%
\end{pmatrix}%
,\ %
\phi_2 \ = \ \begin{pmatrix}
\alpha _{1} & 0 & 0 \\ 
0 & 1 & \alpha _{1}\alpha _{2} \\ 
0 & 1 & 0%
\end{pmatrix}%
.
\end{equation*}%
Then $\theta \ast \phi \in \big\{ \big( 0,\ 0,\ \Delta _{12}\big), \ 
\big(\Delta _{22},\ 0,\ \Delta _{12}\big) \big\} $. 
So we get  ${\rm R}_{27}$ and ${\rm R}_{28}$.

\item $\theta =\eta _{4}$. Let $\phi=\phi_1$   if $\alpha _{1}=0$ and 
$\phi=\phi_2$  if $\alpha _{1}\neq 0$:%
\begin{equation*}
\phi_1 \ = \ \begin{pmatrix}
1 & 0 & 0 \\ 
0 & 1 & 0 \\ 
0 & 0 & \alpha _{2}%
\end{pmatrix}%
, \ 
\phi_2 \ = \ \begin{pmatrix}
\alpha _{1} & 0 & 0 \\ 
0 & 1 & 0 \\ 
0 & 0 & \alpha _{1}\alpha _{2}%
\end{pmatrix}%
.
\end{equation*}%
Then $\theta \ast \phi \in \big\{ \big( 0, \ 0,\ \Delta _{12}\big),\ 
\big(\Delta _{22},\ 0,\ \Delta _{12}\big) \big\} $. So we get   ${\rm R}_{27}$ and ${\rm R}_{28}$.
\end{itemize}
\end{enumerate}

\subsection{Right alternative superalgebras of type $(2,1)$}

\begin{proposition}
Let ${\bf J}$ be a $3$-dimensional Jordan superalgebra of type $\left(
2,1\right) $. Then ${\bf J}$ is isomorphic to one of the following
superalgebras:

\begin{longtable}{l c l l l ll}
    \hline

    ${\bf J}_{01}$ & $:$ & $e_{1}\cdot e_{1}=e_{1}$ &  $e_{2}\cdot e_{2}=e_{2}$  &  &  \\
    ${\bf J}_{02}$ & $:$ & $e_{1}\cdot e_{1}=e_{1}$ &  $e_{2}\cdot e_{2}=e_{2}$  & $e_{1}\cdot f_{1}=f_{1}$ &  &  \\
    ${\bf J}_{03}$ & $:$ & $e_{1}\cdot e_{1}=e_{1}$ & $e_{2}\cdot e_{2}=e_{2}$ & $e_{1}\cdot f_{1}=\frac{1}{2} f_{1}$ &  &  \\
    ${\bf J}_{04}$ & $:$ & $e_{1}\cdot e_{1}=e_{1}$ &  $e_{2}\cdot e_{2}=e_{2}$& $e_{1}\cdot f_{1}=\frac{1}{2} f_{1}$ & $e_{2}\cdot f_{1}=\frac{1}{2} f_{1}$ &  \\
    ${\bf J}_{05}$ & $:$ & $e_{1}\cdot e_{1}=e_{1}$ &  &  &  \\
    ${\bf J}_{06}$ & $:$ & $e_{1}\cdot e_{1}=e_{1}$ && $e_{1}\cdot f_{1}=\frac{1}{2} f_{1}$ &  &  \\
    ${\bf J}_{07}$ & $:$ & $e_{1}\cdot e_{1}=e_{1}$ && $e_{1}\cdot f_{1}=f_{1}$ &  &  \\
    ${\bf J}_{08}$ & $:$ & $e_{1}\cdot e_{1}=e_{1}$ & $e_{1}\cdot e_{2}=e_{2}$ &  &  \\
    ${\bf J}_{09}$ & $:$ & $e_{1}\cdot e_{1}=e_{1}$ & $e_{1}\cdot e_{2}=e_{2}$ & $e_{1}\cdot f_{1}=\frac{1}{2} f_{1}$ &  \\
    ${\bf J}_{10}$ & $:$ & $e_{1}\cdot e_{1}=e_{1}$ & $e_{1}\cdot e_{2}=e_{2}$ & $e_{1}\cdot f_{1}=f_{1}$ &  \\
    ${\bf J}_{11}$ & $:$ & $e_{1}\cdot e_{1}=e_{1}$ & $e_{1}\cdot e_{2}=\frac{1}{2} e_{2}$ &  &  \\
    ${\bf J}_{12}$ & $:$ & $e_{1}\cdot e_{1}=e_{1}$ & $e_{1}\cdot e_{2}=\frac{1}{2} e_{2}$ & $e_{1}\cdot f_{1}=\frac{1}{2} f_{1}$ &  \\
    ${\bf J}_{13}$ & $:$ & $e_{1}\cdot e_{1}=e_{1}$ & $e_{1}\cdot e_{2}=\frac{1}{2} e_{2}$ & $e_{1}\cdot f_{1}=f_{1}$ &  \\
    ${\bf J}_{14}$ & $:$ & $e_{1}\cdot e_{1}=e_{2}$ &  &  &  \\
    ${\bf J}_{15}$ & $:$ & \text{trivial} \\
    \hline
\end{longtable}

\end{proposition}

\subsubsection{The classification Theorem A2}
\begin{theoremA2}
Let ${\rm R}$ be a nontrivial $3$-dimensional right alternative
superalgebra of type $\left( 2,1\right) $. Then ${\rm R}$ is isomorphic
to one of the following superalgebras:

\begin{longtable}{llllllllllllllllllll}
${\bf R}_{01}$&$:$&$e_{1}e_{1}=e_{1}$&$e_{2}e_{2}=e_{2}$\\

${\bf R}_{02}$&$:$&$e_{1}e_{1}=e_{1}$&$e_{2}e_{2}=e_{2}$&$e_{1}f_{1}=f_{1}$&$f_{1}e_{1}=f_{1}$\\

${\bf R}_{03}$&$:$&$e_{1}e_{1}=e_{1}$&$e_{2}e_{2}=e_{2}$&$e_{1}f_{1}=f_{1}$&$f_{1}e_{1}=f_{1}$&$f_{1}f_{1}=e_{1}$\\
 
${\bf R}_{04}$&$:$&$e_{1}e_{1}=e_{1}$&$e_{2}e_{2}=e_{2}$&$f_{1}e_{1}=f_{1}$ \\

${\bf R}_{05}$&$:$&$e_{1}e_{1}=e_{1}$&$e_{2}e_{2}=e_{2}$&$e_{1}f_{1}=f_{1}$\\

${\bf R}_{06}$&$:$&$e_{1}e_{1}=e_{1}$&$e_{2}e_{2}=e_{2}$&$e_{1}f_{1}=f_{1}$&$f_{1}f_{1}=e_{1}$\\

${\bf R}_{07}$&$:$&$e_{1}e_{1}=e_{1}$&$e_{2}e_{2}=e_{2}$&$e_{2}f_{1}=f_{1}$&$f_{1}e_{1}=f_{1}$\\

${\bf R}_{08}$&$:$&$e_{1}e_{1}=e_{1}$&$e_{2}e_{2}=e_{2}$&$e_{2}f_{1}=f_{1}$&$f_{1}e_{1}=f_{1}$&$f_{1}f_{1}=e_{2}$\\

${\bf R}_{09}$&$:$&$e_{1}e_{1}=e_{1}$\\

${\bf R}_{10}$&$:$&$e_{1}e_{1}=e_{1}$&$f_{1}f_{1}=e_{2}$\\

${\bf R}_{11}$&$:$&$e_{1}e_{1}=e_{1}$&$f_{1}e_{1}=f_{1}$\\

${\bf R}_{12}$&$:$&$e_{1}e_{1}=e_{1}$&$f_{1}e_{1}=f_{1}$&$f_{1}f_{1}=e_{2}$\\

${\bf R}_{13}$&$:$&$e_{1}e_{1}=e_{1}$&$e_{1}f_{1}=f_{1}$\\

${\bf R}_{14}$&$:$&$e_{1}e_{1}=e_{1}$&$e_{1}f_{1}=f_{1}$&$f_{1}f_{1}=e_{1}$\\

${\bf R}_{15}$&$:$&$e_{1}e_{1}=e_{1}$&$e_{1}f_{1}=f_{1}$&$f_{1}e_{1}=f_{1}$\\

${\bf R}_{16}$&$:$&$e_{1}e_{1}=e_{1}$&$e_{1}f_{1}=f_{1}$&$f_{1}e_{1}=f_{1}$&$f_{1}f_{1}=e_{1}$\\

${\bf R}_{17}$&$:$&$e_{1}e_{1}=e_{1}$&$e_{1}e_{2}=e_{2}$&$e_{2}e_{1}=e_{2}$\\

${\bf R}_{18}$&$:$&$e_{1}e_{1}=e_{1}$&$e_{1}e_{2}=e_{2}$&$e_{2}e_{1}=e_{2}$&$e_{1}f_{1}=f_{1}$\\

${\bf R}_{19}$&$:$&$e_{1}e_{1}=e_{1}$&$e_{1}e_{2}=e_{2}$&$e_{2}e_{1}=e_{2}$&$e_{1}f_{1}=f_{1}$&$f_{1}f_{1}=e_{2}$\\

${\bf R}_{20}$&$:$&$e_{1}e_{1}=e_{1}$&$e_{1}e_{2}=e_{2}$&$e_{2}e_{1}=e_{2}$&$f_{1}e_{1}=f_{1}$\\

${\bf R}_{21}$&$:$&$e_{1}e_{1}=e_{1}$&$e_{1}e_{2}=e_{2}$&$e_{2}e_{1}=e_{2}$&$e_{1}f_{1}=f_{1}$&$f_{1}e_{1}=f_{1}$\\

${\bf R}_{22}$&$:$&$e_{1}e_{1}=e_{1}$&$e_{1}e_{2}=e_{2}$&$e_{2}e_{1}=e_{2}$&$e_{1}f_{1}=f_{1}$&$f_{1}e_{1}=f_{1}$&$f_{1}f_{1}=e_{2}$\\

${\bf R}_{23}$&$:$&$e_{1}e_{1}=e_{1}$&$e_{1}e_{2}=e_{2}$\\

${\bf R}_{24}$&$:$&$e_{1}e_{1}=e_{1}$&$e_{1}e_{2}=e_{2}$&$f_{1}f_{1}=e_{2}$\\

${\bf R}_{25}$&$:$&$e_{1}e_{1}=e_{1}$&$e_{2}e_{1}=e_{2}$\\

${\bf R}_{26}$&$:$&$e_{1}e_{1}=e_{1}$&$e_{2}e_{1}=e_{2}$&$e_{1}f_{1}=f_{1}$\\

${\bf R}_{27}$&$:$&$e_{1}e_{1}=e_{1}$&$e_{2}e_{1}=e_{2}$&$e_{1}f_{1}=f_{1}$&$f_{1}f_{1}=e_{1}$\\

${\bf R}_{28}$&$:$&$e_{1}e_{1}=e_{1}$&$e_{2}e_{1}=e_{2}$&$e_{1}f_{1}=f_{1}$&$f_{1}f_{1}=e_{2}$\\

${\bf R}_{29}$&$:$&$e_{1}e_{1}=e_{1}$&$e_{1}e_{2}=e_{2}$&$f_{1}e_{1}=f_{1}$\\

${\bf R}_{30}$&$:$&$e_{1}e_{1}=e_{1}$&$e_{1}e_{2}=e_{2}$&$f_{1}e_{1}=f_{1}$&$f_{1}f_{1}=e_{2}$\\

${\bf R}_{31}$&$:$&$e_{1}e_{1}=e_{1}$&$e_{2}e_{1}=e_{2}$&$f_{1}e_{1}=f_{1}$\\

${\bf R}_{32}$&$:$&$e_{1}e_{1}=e_{1}$&$e_{1}e_{2}=e_{2}$&$e_{1}f_{1}=f_{1}$\\

${\bf R}_{33}$&$:$&$e_{1}e_{1}=e_{1}$&$e_{2}e_{1}=e_{2}$&$e_{1}f_{1}=f_{1}$&$f_{1}e_{1}=f_{1}$\\

${\bf R}_{34}$&$:$&$e_{1}e_{1}=e_{1}$&$e_{2}e_{1}=e_{2}$&$e_{1}f_{1}=f_{1}$&$f_{1}e_{1}=f_{1}$&$f_{1}f_{1}=e_{1}$\\

${\bf R}_{35}$&$:$&$e_{1}e_{1}=e_{1}$&$e_{2}e_{1}=e_{2}$&$e_{1}f_{1}=f_{1}$&$f_{1}e_{1}=f_{1}$&$f_{1}f_{1}=e_{2}$\\

${\bf R}_{36}$&$:$&$e_{1}e_{1}=e_{1}$&$e_{1}e_{2}=e_{2}$&$e_{1}f_{1}=f_{1}$&$f_{1}e_{1}=f_{1}$\\

${\bf R}_{37}$&$:$&$e_{1}e_{1}=e_{2}$\\

${\bf R}_{38}$&$:$&$e_{1}e_{1}=e_{2}$&$f_{1}f_{1}=e_{2}$\\

${\bf R}_{39}$&$:$&$f_{1}f_{1}=e_{1}$
\end{longtable}
\end{theoremA2}

\subsubsection{The proof of  Theorem A2}
Let ${\bf R}$ be a nontrivial $3$-dimensional right alternative
superalgebra of type $\left( 2,1\right) $. Then ${\bf R}^{+}$ is a $3$%
-dimensional Jordan superalgebra of type $\left( 2,1\right) $. Then we may
assume ${\bf R}^{+}\in \left\{ {\bf J}_{01},\ldots ,{\bf J}_{15}\right\} $. So we have the following cases:

\begin{enumerate}[I.]
    \item 
\underline{${\bf R}^{+}={\bf J}_{01}$}. Then ${\rm Z}^{2}\left( {\bf J}_{01},{\bf J}_{01}\right) =\left\{ 0\right\} $. So we get the superalgebra $%
{\bf R}_{01}$.

 \item \underline{${\bf R}^{+}={\bf J}_{02}$}. Let $\theta =
 \big(
{\rm B}_{1},\ {\rm B}_{2},\ {\rm B}_{3}\big) $ be an arbitrary element of ${\rm Z}^{2}\left( \bf{J}_{02},{\bf J}_{02}\right) $. Then $\theta \ =\ \big( \alpha _{1}\Delta
_{33},\ 0,\ 0\big) $ for some $\alpha _{1}\in \mathbb{C}$. The automorphism
group  ${\rm Aut}\left( {\bf J}_{02}\right) $ 
consists of the automorphisms $\phi $ given by a matrix of the following
form:%
\begin{equation*}
\phi =%
\begin{pmatrix}
1 & 0 & 0 \\ 
0 & 1 & 0 \\ 
0 & 0 & a_{33}%
\end{pmatrix}%
.
\end{equation*}%
Let $\phi =\bigl(a_{ij}\bigr)\in $ ${\rm Aut}\left( {\bf J}_{02}\right) 
$. Then $\theta \ast \phi =\left( \alpha _{1}a_{33}^{2}\Delta
_{33},\ 0,\ 0\right) $. Hence we have the representatives 
$\big( 0,\ 0,\ 0\big)
$ and $\big( \Delta _{33},\ 0,\ 0\big)$. 
Thus we get the superalgebras ${\bf R}_{02} $ and ${\bf R}_{03}$.

   \item \underline{${\bf R}^{+}={\bf J}_{03}$}. Let $\theta =\big(
{\rm B}_{1},\ {\rm B}_{2},\ {\rm B}_{3}\big) $ be an arbitrary element of ${\rm Z}^{2}\left({\bf J}_{03},{\bf J}_{03}\right) $. Then  $\theta =\left( 0,\ 0,\ - \frac{1}{2}\Delta _{13}\right) $ and  $\theta =\left( \alpha _{1}\Delta
_{33},\ 0,\ \frac{1}{2}\Delta _{13}\right) $ for some $\alpha _{1}\in \mathbb{C}$%
. The automorphism group  ${\rm Aut}\left( {\bf J}_{03}\right) $  consists of the automorphisms $\phi $ given by a matrix of
the following form:%
\begin{equation*}
\phi =%
\begin{pmatrix}
1 & 0 & 0 \\ 
0 & 1 & 0 \\ 
0 & 0 & a_{33}%
\end{pmatrix}%
.
\end{equation*}%
Let $\phi =\bigl(a_{ij}\bigr)\in $ ${\rm Aut}\left( {\bf J}_{03}\right) 
$. Then $\theta \ast \phi =\left( \alpha _{1}a_{33}^{2}\Delta _{33},\ 0,\ \frac{1}{2}\Delta _{13}\right).$ 
Hence, we have the representatives  $\left( 0,\ 0,\ -\frac{1}{2}\Delta _{13}\right),$ 
$\left( 0,\ 0,\ \frac{1}{2}\Delta _{13}\right)$ and 
$\left( \Delta _{33},\ 0,\ \frac{1}{2}\Delta_{13}\right).$ 
Thus we get the superalgebras  ${\bf R}_{04},$   ${\bf R}_{05}$ and ${\bf R}_{06}$.

   \item \underline{${\bf R}^{+}={\bf J}_{04}$}. Let $\theta =\big(
{\rm B}_{1},\ {\rm B}_{2},\ {\rm B}_{3}\big) $ be an arbitrary element of ${\rm Z}^{2}\left( {\bf J}_{04},{\bf J}_{04}\right) $. Then $\theta \in \left\{ \eta _{1},\eta
_{2}\right\} $ where%
\begin{longtable}{lcllcl}
$\eta _{1} $&$=$&$\left( \alpha _{1}\Delta _{33},\ 0,\ \frac{1}{2}\Delta _{13}-\frac{1}{2}\Delta _{23}\right),$ &
$\eta _{2} $&$=$&$\left( 0,\ \alpha _{1}\Delta _{33},\ -\frac{1}{2}\Delta _{13}+\frac{1}{2}\Delta _{23}\right),$
\end{longtable}%
for some $\alpha _{1}\in \mathbb{C}$. 
The automorphism group  ${\rm Aut}\left( {\bf J}_{04}\right) $  consists of the
automorphisms $\phi $ given by a matrix of the following form:%
\begin{equation*}
\begin{pmatrix}
1 & 0 & 0 \\ 
0 & 1 & 0 \\ 
0 & 0 & a_{33}%
\end{pmatrix}%
,\ %
\begin{pmatrix}
0 & 1 & 0 \\ 
1 & 0 & 0 \\ 
0 & 0 & a_{33}%
\end{pmatrix}%
.
\end{equation*}

\begin{itemize}
\item $\theta =\eta _{1}$. If $\alpha _{1}=0$, we get the superalgebra ${\bf R}_{07}$. If $\alpha _{1}\neq 0$, we choose $\phi $ to be the following
automorphism:%
\begin{equation*}
\phi =%
\begin{pmatrix}
1 & 0 & 0 \\ 
0 & 1 & 0 \\ 
0 & 0 & \frac{1}{\sqrt{\alpha _{1}}}%
\end{pmatrix}%
.
\end{equation*}%
Then $\theta \ast \phi =\left( \Delta _{33},\ 0,\ \frac{1}{2}\Delta _{13}-\frac{1%
}{2}\Delta _{23}\right) $. Hence we obtain the superalgebra ${\bf R}_{08}$.

\item $\theta =\eta _{2}$. We choose $\phi $ to be the following
automorphism:%
\begin{equation*}
\phi =%
\begin{pmatrix}
0 & 1 & 0 \\ 
1 & 0 & 0 \\ 
0 & 0 & 1%
\end{pmatrix}%
.
\end{equation*}%
Then $\eta _{2}\ast \phi =\eta _{1}$.
\end{itemize}

   \item \underline{${\rm R}^{+}={\bf J}_{05}$}. Let 
   $\theta =\big( {\rm B}_{1},\ {\rm B}_{2},\ {\rm B}_{3}\big) $ be an arbitrary element of ${\rm Z}^{2}\left( {\bf J}_{05},{\bf J}_{05}\right) $. Then $\theta =\big( 0,\ \alpha _{1}\Delta_{33},\ 0\big) $ for some $\alpha _{1}\in \mathbb{C}$. The automorphism
group   ${\rm Aut}\left( {\bf J}_{05}\right) $ 
consists of the automorphisms $\phi $ given by a matrix of the following
form:%
\begin{equation*}
\phi =%
\begin{pmatrix}
1 & 0 & 0 \\ 
0 & a_{22} & 0 \\ 
0 & 0 & a_{33}%
\end{pmatrix}%
.
\end{equation*}%
Let $\phi =\bigl(a_{ij}\bigr)\in $ ${\rm Aut}\left( {\bf J}_{05}\right) 
$. 
Then $\theta \ast \phi =\left( 0,\ {\alpha _{1}}{a^{-1}_{22}}a_{33}^{2}\Delta _{33},\ 0\right) $. Hence we have the representatives $\big(0,\ 0,\ 0\big)$ and $\big( 0,\ \Delta _{33},\ 0\big).$ Thus we get   ${\bf R}_{09}$ and ${\bf R}_{10}$.

   \item \underline{${\bf R}^{+}={\bf J}_{06}$}. Let $\theta =\big(
{\rm B}_{1},\ {\rm B}_{2},\ {\rm B}_{3}\big) $ be an arbitrary element of ${\rm Z}^{2}\left( {\bf J}_{06},{\bf J}_{06}\right) $. Then $\theta $ $\in \left\{ \eta _{1},\eta
_{2}\right\} $ where%
\begin{longtable}{lcllcl}
$\eta _{1} $&$=$&$\left( 0,\ \alpha _{1}\Delta _{33},\ -\frac{1}{2}\Delta_{13}\right),$ & 
$\eta _{2} $&$=$&$\left( \alpha _{1}\Delta _{33},\ 0,\ \frac{1}{2}\Delta _{13}\right)
,$
\end{longtable}
for some $\alpha _{1}\in \mathbb{C}$. The automorphism group   ${\rm Aut}\left( {\bf J}_{06}\right) $  consists of the
automorphisms $\phi $ given by a matrix of the following form:%
\begin{equation*}
\phi =%
\begin{pmatrix}
1 & 0 & 0 \\ 
0 & a_{22} & 0 \\ 
0 & 0 & a_{33}%
\end{pmatrix}%
.
\end{equation*}

\begin{itemize}
\item $\theta =\eta _{1}$. Let $\phi =\bigl(a_{ij}\bigr)\in $ ${\rm Aut}%
\left( {\bf J}_{06}\right) $. Then 
$\theta \ast \phi\ =\ \left( 0,\ {\alpha _{1}}{a^{-1}_{22}}a_{33}^{2}\Delta _{33},\ -\frac{1}{2}\Delta _{13}\right).$ 
Hence we have the representatives $\left( 0,\ 0,\ -\frac{1}{2}\Delta_{13}\right)$ and 
$\left( 0,\ \Delta _{33},\ -\frac{1}{2}\Delta _{13}\right).$ 
Thus we get the superalgebras ${\bf R}_{11}$ and ${\bf R}_{12}$.

\item $\theta =\eta _{2}$. Let $\phi =\bigl(a_{ij}\bigr)\in $ ${\rm Aut}%
\left( {\bf J}_{06}\right) $. Then 
$\theta \ast \phi =\left( \alpha_{1}a_{33}^{2}\Delta _{33},\ 0,\ \frac{1}{2}\Delta _{13}\right).$ Hence we have
the representatives $\left( 0,\ 0,\ \frac{1}{2}\Delta _{13}\right)$ and 
$\left(\Delta _{33},\ 0,\ \frac{1}{2}\Delta _{13}\right).$ 
Thus we get the superalgebras ${\bf R}_{13}$ and ${\bf R}_{14}.$
\end{itemize}

   \item \underline{${\bf R}^{+}={\bf J}_{07}$}. Let $\theta =\big(
{\rm B}_{1},\ {\rm B}_{2},\ {\rm B}_{3}\big) $ be an arbitrary element of ${\rm Z}^{2}\left( {\bf J}_{07},{\bf J}_{07}\right) $. Then 
$\theta =\big( \alpha _{1}\Delta_{33},\ 0,\ 0\big) $ for some $\alpha _{1}\in \mathbb{C}$. The automorphism   ${\rm Aut}\left( {\bf J}_{07}\right) $ 
consists of the automorphisms $\phi $ given by a matrix of the following
form:%
\begin{equation*}
\phi =%
\begin{pmatrix}
1 & 0 & 0 \\ 
0 & a_{22} & 0 \\ 
0 & 0 & a_{33}%
\end{pmatrix}%
.
\end{equation*}%
Let $\phi =\bigl(a_{ij}\bigr)\in $ ${\rm Aut}\left( {\bf J}_{07}\right) 
$. Then $\theta \ast \phi =\left( \alpha _{1}a_{33}^{2}\Delta
_{33},\ 0,\ 0\right) $. Hence we have the representatives 
$\big( 0,\ 0,\ 0\big)$ and $\big( \Delta _{33},0,0\big) $. 
Thus we get the superalgebras ${\bf R}_{15}$ and ${\bf R}_{16}$.

   \item \underline{${\bf R}^{+}={\bf J}_{08}$}. Then ${\rm Z}^{2}\left( {\bf J}_{08},{\bf J}_{08}\right) =\left\{ 0\right\} $. So we get the superalgebra $%
{\bf R}_{17}$.

   \item \underline{${\bf R}^{+}={\bf J}_{09}$}. Let $\theta =\big(
{\rm B}_{1},\ {\rm B}_{2},\ {\rm B}_{3}\big) $ 
be an arbitrary element of ${\rm Z}^{2}\left( {\bf J}_{09},{\bf J}_{09}\right) $. Then $\theta $ $\in \left\{ \eta _{1},\eta
_{2}\right\} $ where%
\begin{longtable}{lcllcl}
$\eta _{1} $&$=$&$\left( 0,\ \alpha _{1}\Delta _{33},\ \frac{1}{2}\Delta _{13}\right),$&$
\eta _{2} $&$=$&$\left( 0,\ 0,\ -\frac{1}{2}\Delta _{13}\right),$
\end{longtable}%
for some $\alpha _{1}\in \mathbb{C}$. The automorphism group    ${\rm Aut}\left( {\bf J}_{09}\right) $  consists of the
automorphisms $\phi $ given by a matrix of the following form
\begin{equation*}
\phi =%
\begin{pmatrix}
1 & 0 & 0 \\ 
0 & a_{22} & 0 \\ 
0 & 0 & a_{33}%
\end{pmatrix}%
.
\end{equation*}

\begin{itemize}
\item $\theta =\eta _{1}$. Let $\phi =\bigl(a_{ij}\bigr)\in $ ${\rm Aut}%
\left( {\bf J}_{09}\right) $. Then 
$\theta \ast \phi =\left( 0,\ {\alpha _{1}}{a^{-1}_{22}}a_{33}^{2}\Delta _{33},\ \frac{1}{2}\Delta _{13}\right).$
Hence we have representatives 
$\left( 0,\ 0,\ \frac{1}{2}\Delta _{13}\right)$ and 
$\left( 0,\ \Delta _{33},\ \frac{1}{2}\Delta _{13}\right) $. Thus we get the
superalgebras ${\bf R}_{18}$ and ${\bf R}_{19}$.

\item $\theta =\eta _{2}$. We get the superalgebra ${\bf R}_{20}$.
\end{itemize}

   \item \underline{${\bf R}^{+}={\bf J}_{10}$}. Let $\theta =\big(
{\rm B}_{1},\ {\rm B}_{2},\ {\rm B}_{3}\big) $ be an arbitrary element of ${\rm Z}^{2}\left( {\bf J}_{10},{\bf J}_{10}\right) $. Then 
$\theta =\big( 0,\ \alpha _{1}\Delta_{33},\ 0\big) $ for some $\alpha _{1}\in \mathbb{C}$. The automorphism group 
   ${\rm Aut}\left( {\bf J}_{10}\right) $ 
consists of the automorphisms $\phi $ given by a matrix of the following
form:%
\begin{equation*}
\phi =%
\begin{pmatrix}
1 & 0 & 0 \\ 
0 & a_{22} & 0 \\ 
0 & 0 & a_{33}%
\end{pmatrix}%
.
\end{equation*}%
Let $\phi =\bigl(a_{ij}\bigr)\in $ ${\rm Aut}\left( {\bf J}_{10}\right) 
$. Then 
$\theta \ast \phi =\left( 0,\ {\alpha _{1}}{a^{-1}_{22}}a_{33}^{2}\Delta _{33},\ 0\right).$ 
Hence, we have the representatives $\big(
0,\ 0,\ 0\big)$ and $\big( 0,\ \Delta _{33},\ 0\big) $. Thus, we get $%
{\bf R}_{21}$ and ${\bf R}_{22}$.

\item 
\underline{${\bf R}^{+}={\bf J}_{11}$}. Let $\theta =\big(
{\rm B}_{1},\ {\rm B}_{2},\ {\rm B}_{3}\big) $ be an arbitrary element of ${\rm Z}^{2}\left( {\bf J}_{11},{\bf J}_{11}\right) $. Then $\theta $ $\in \left\{ \eta _{1},\eta
_{2}\right\},$ where 
\begin{longtable}{lcllcl}
$\eta _{1} $&$=$&$\left( 0,\frac{1}{2}\Delta _{12}+\alpha _{1}\Delta
_{33},0\right),$ &
$\eta _{2} $&$=$&$\left( 0,-\frac{1}{2}\Delta _{12},0\right),$
\end{longtable}%
for some $\alpha _{1}\in \mathbb{C}$. The automorphism group  ${\rm Aut}\left( {\bf J}_{11}\right) $ consists of the
automorphisms $\phi $ given by a matrix of the following form:%
\begin{equation*}
\phi =%
\begin{pmatrix}
1 & 0 & 0 \\ 
a_{21} & a_{22} & 0 \\ 
0 & 0 & a_{33}%
\end{pmatrix}%
.
\end{equation*}

\begin{itemize}
\item $\theta =\eta _{1}$. Let $\phi =\bigl(a_{ij}\bigr)\in $ ${\rm Aut}%
\left( {\bf J}_{11}\right) $. Then 
$\theta \ast \phi =\left( 0,\ \frac{1}{2}\Delta _{12}+{\alpha _{1}}{a^{-1}_{22}}a_{33}^{2}\Delta _{33},\ 0\right) $.
Hence we have the representatives 
$\left( 0,\ \frac{1}{2}\Delta _{12},\ 0\right)$ and 
$\left( 0,\ \frac{1}{2}\Delta _{12}+\Delta _{33},\ 0\right).$ 
Thus we get the superalgebras ${\bf R}_{23}$ and ${\bf R}_{24}$.

\item $\theta =\eta _{2}$. We get the superalgebra ${\bf R}_{25}$.
\end{itemize}

   \item \underline{${\bf R}^{+}={\bf J}_{12}$}. Let $\theta =\big(
{\rm B}_{1},\ {\rm B}_{2},\ {\rm B}_{3}\big) $ be an arbitrary element of ${\rm Z}^{2}\left({\bf J}_{12},{\bf J}_{12}\right) $. Then $\theta $ $\in \left\{ \eta
_{1},\ldots ,\eta _{4}\right\},$ where 
\begin{longtable}{lcllcl}
$\eta _{1} $&$=$&$\left( \alpha _{1}\Delta _{33},\ -\frac{1}{2}\Delta _{12}+\alpha_{2}\Delta _{33},\ \frac{1}{2}\Delta _{13}\right) ,$ & 
$\eta _{2} $&$=$&$\left( 0,\ \frac{1}{2}\Delta _{12}+\alpha _{1}\Delta _{33},\ -\frac{1}{2}\Delta _{13}\right) ,$ \\
$\eta _{3} $&$=$&$\left( 0,\ -\frac{1}{2}\Delta _{12},\ -\frac{1}{2}\Delta_{13}\right),$ & 
$\eta _{4} $&$=$&$\left( 0,\ \frac{1}{2}\Delta _{12},\ \frac{1}{2}\Delta _{13}\right),$
\end{longtable}
for some $\alpha _{1},\alpha _{2}\in \mathbb{C}$. The automorphism group   ${\rm Aut}\left( {\bf J}_{12}\right) $ consists of
the automorphisms $\phi $ given by a matrix of the following form:%
\begin{equation*}
\phi =%
\begin{pmatrix}
1 & 0 & 0 \\ 
a_{21} & a_{22} & 0 \\ 
0 & 0 & a_{33}%
\end{pmatrix}%
.
\end{equation*}

\begin{itemize}
\item $\theta =\eta _{1}$. Let $\phi =\bigl(a_{ij}\bigr)\in $ ${\rm Aut}%
\left( {\bf J}_{12}\right) $. Then 
\begin{center}$\theta \ast \phi \ =\ \left( \alpha
_{1}a_{33}^{2}\Delta _{33},\ -\frac{1}{2}\Delta _{12}+{a^{-1}_{22}}a_{33}^{2}\left( \alpha _{2}-\alpha _{1}a_{21}\right) \Delta _{33},\ \frac{1}{2}\Delta _{13}\right) $.
\end{center} Hence we have the representatives 
$\left( 0,\ -\frac{1}{2}\Delta _{12},\ \frac{1}{2}\Delta _{13}\right),$ \ 
$\left( \Delta _{33},\ -\frac{1}{2}\Delta _{12},\ \frac{1}{2}\Delta _{13}\right),$ and 
$\left( 0,\ -\frac{1}{2}\Delta_{12}+\Delta _{33},\ \frac{1}{2}\Delta _{13}\right).$ Thus we get the superalgebras ${\bf R}_{26},$ \ ${\bf R}_{27},$ and ${\bf R}_{28}$.

\item $\theta =\eta _{2}$. Let $\phi =\bigl(a_{ij}\bigr)\in $ ${\rm Aut}%
\left( {\bf J}_{12}\right) $. Then 
$\theta \ast \phi =\left( 0,\ \frac{1}{2}\Delta _{12}+{\alpha _{1}}{a^{-1}_{22}}a_{33}^{2}\Delta _{33}, \ -\frac{1}{2}\Delta _{13}\right).$ 
Hence we have representatives 
$\left( 0,\ \frac{1}{2}\Delta _{12},\ -\frac{1}{2}\Delta _{13}\right)$ and  
$\left( 0,\ \frac{1}{2}\Delta_{12}+\Delta _{33},\ -\frac{1}{2}\Delta _{13}\right).$ 
Thus we get the
superalgebras ${\bf R}_{29}$ and ${\bf R}_{30}$.

\item $\theta =\eta _{3}$. We get the superalgebra ${\bf R}_{31}$.

\item $\theta =\eta _{4}$. We get the superalgebra ${\bf R}_{32}$.
\end{itemize}

   \item \underline{${\bf R}^{+}={\bf J}_{13}$}. Let $\theta =\big(
{\rm B}_{1},\ {\rm B}_{2},\ {\rm B}_{3}\big) $ be an arbitrary element of ${\rm Z}^{2}\left( {\bf J}_{13},{\bf J}_{13}\right) $. Then $\theta $ $\in \left\{ \eta _{1},\eta
_{2}\right\} $ where 
\begin{longtable}{lcllcl}
$\eta _{1} $&$=$&$\left( \alpha _{1}\Delta _{33},\ -\frac{1}{2}\Delta _{12}+\alpha
_{2}\Delta _{33},\ 0\right),$ &
$\eta _{2} $&$=$&$\left( 0,\ \frac{1}{2}\Delta _{12},\ 0\right),$
\end{longtable}
for some $\alpha _{1},\alpha _{2}\in \mathbb{C}$. The automorphism group    ${\rm Aut}\left( {\bf J}_{13}\right) $  consists of
the automorphisms $\phi $ given by a matrix of the following form:%
\begin{equation*}
\phi =%
\begin{pmatrix}
1 & 0 & 0 \\ 
a_{21} & a_{22} & 0 \\ 
0 & 0 & a_{33}%
\end{pmatrix}%
.
\end{equation*}

\begin{itemize}
\item $\theta =\eta _{1}$. Let $\phi =\bigl(a_{ij}\bigr)\in $ ${\rm Aut}%
\left( {\bf J}_{13}\right) $. Then 
\begin{center}
$\theta \ast \phi \ =\ \left( \alpha
_{1}a_{33}^{2}\Delta _{33},\ -\frac{1}{2}\Delta _{12}+{a^{-1}_{22}}a_{33}^{2}\left( \alpha _{2}-\alpha _{1}a_{21}\right) \Delta _{33},\ 0\right).$
\end{center}
Hence we have the representatives 
$\left( 0,\ -\frac{1}{2}\Delta_{12},\ 0\right),$ 
$\left( \Delta _{33},\ -\frac{1}{2}\Delta _{12},\ 0\right),$ and 
$\left( 0,\ -\frac{1}{2}\Delta _{12}+\Delta _{33},\ 0\right) $. Thus we get the
superalgebras ${\bf R}_{33},$ ${\bf R}_{34},$ and ${\bf R}_{35}$.

\item $\theta =\eta _{2}$. We get the superalgebra ${\bf R}_{36}$.
\end{itemize}

   \item \underline{${\bf R}^{+}={\bf J}_{14}$}. Let $\theta =\big(
{\rm B}_{1},\ {\rm B}_{2},\ {\rm B}_{3}\big) $ be an arbitrary element of ${\rm Z}^{2}\left({\bf J}_{14},{\bf J}_{14}\right) $. Then $\theta =\big( 0,\ \alpha _{1}\Delta
_{33},\ 0\big) $ for some $\alpha _{1}\in \mathbb{C}$. The automorphism
group   ${\rm Aut}\left( {\bf J}_{14}\right) $ 
consists of the automorphisms $\phi $ given by a matrix of the following
form:%
\begin{equation*}
\phi =%
\begin{pmatrix}
a_{11} & 0 & 0 \\ 
a_{21} & a_{11}^{2} & 0 \\ 
0 & 0 & a_{33}%
\end{pmatrix}%
.
\end{equation*}%
Let $\phi =\bigl(a_{ij}\bigr)\in $ ${\rm Aut}\left( {\bf J}_{14}\right) 
$. Then 
$\theta \ast \phi =\left( 0,\ {\alpha _{1}}{a_{11}^{-2}}a_{33}^{2}\Delta _{33},\ 0\right) $. 
Hence, we have the representatives 
$\big(0,\ 0,\ 0\big)$ and 
$\big( 0,\ \Delta _{33},\ 0\big) $. Thus, we get  $%
{\bf R}_{37}$ and ${\bf R}_{38}$. 

   \item \underline{${\bf R}^{+}={\bf J}_{15}$}. Let $\theta =\big(
{\rm B}_{1},\ {\rm B}_{2},\ {\rm B}_{3}\big) \neq 0$ be an arbitrary element of ${\rm Z}^{2}\left( 
{\bf J}_{15},{\bf J}_{15}\right) $. Then 
$\theta =\big( \alpha_{1}\Delta _{33},\ \alpha _{2}\Delta _{33},\ 0\big)$ 
for some $\alpha_{1},\alpha _{2},\alpha _{3}\in \mathbb{C}$. The automorphism group    ${\rm Aut}\left( {\bf J}_{15}\right) $  consists of the automorphisms $\phi$ given by a matrix of the following form: %
\begin{equation*}
\phi =%
\begin{pmatrix}
a_{11} & a_{12} & 0 \\ 
a_{21} & a_{22} & 0 \\ 
0 & 0 & a_{33}%
\end{pmatrix}%
.
\end{equation*}%
Let $\phi =\bigl(a_{ij}\bigr)\in $ ${\rm Aut}\left( {\bf J}_{15}\right) 
$. Then $\theta \ast \phi =\left( \beta _{1}\Delta _{33},\ \beta _{2}\Delta_{33},\ 0\right) $ where%
\begin{longtable}{lcllcl}
$\beta _{1} $&$=$&$\frac{a_{33}^{2}(\alpha_{1}a_{22}-\alpha _{2}a_{12}) }{a_{11}a_{22}-a_{12}a_{21}},$ & 
$\beta _{2} $&$=$&$-\frac{a_{33}^{2}(\alpha_{1}a_{21}-\alpha _{2}a_{11})}{a_{11}a_{22}-a_{12}a_{21}}.$
\end{longtable}
Since $\left( \alpha _{1},\alpha _{2}\right) \neq \left( 0,0\right) $, we
may assume $\left( \alpha _{1},\alpha _{2}\right) =\left( 1,0\right) $. So
we get  ${\bf R}_{39}$.
 
\end{enumerate}

 \subsection{Corollaries: $\Omega$- and $\big($binary-$\Omega\big)$-superalgebras}

Let us remember that each $3$-dimensional binary associative (=alternative) algebra is associative \cite{G62};
each $3$-dimensional binary Lie algebras is Lie   \cite{ikp20};
each $3$-dimensional binary $\mathfrak{perm}$ algebra is a $\mathfrak{perm}$ algebra \cite{akl}; 
 each $3$-dimensional binary $\big( -1,1\big)$-algebra is a  $\big( -1,1\big)$-algebra \cite{akl}.   
 The present observation raised the following question.
\medskip 

\noindent{\bf Open question.}
Is it true that each $3$-dimensional binary-$\Omega$\footnote{
By a binary-$\Omega$ algebra we mean an algebra such that each $2$-generated subalgebra is an $\Omega$ algebra for a family of polynomial identities $\Omega$.}
 algebra is an $\Omega$ algebra?

\medskip

It is known that the super-version of the present question is not true:
there is a $2$-dimensional non-Lie binary Lie superalgebra \cite{GRS}.
The aim of the present section is to study super-analogues of the present open question for 
$\mathfrak{perm}$, associative and $\big(-1,1\big)$-superalgebras.
Let us also remember the following inclusions for algebras:

  \medskip

\begin{center}

\begin{tikzpicture}[node distance=1.5cm, auto, text=black] 
\node (Asc) at (0,0) {AssCom}; 
\node at (1.15, 0) {$\subset$}; 

\node (perm) at (2, 0) {$\mathfrak{perm}$}; 
\node at (2.7, 0.3) {\rotatebox{25}{$\subset$}}; 
\node at (2.7, -0.3) {\rotatebox{-25}{$\subset$}};

\node (ass) at (4, 0.5) {Associative}; 
\node (bp) at (4, -0.5) {binary $\mathfrak{perm}$}; 

\node at (5.5, 0.5) {$\subset$}; 
\node at (5.5, -0.5) {$\subset$}; 
\node at (5.5, 0) {\rotatebox{-25}{$\subset$}};

\node (-11) at (7.0, 0.5) {$(-1,1)$-}; 
\node (ncj) at (7.0, -0.5) {binary ass.}; 

\node at (8.3, 0.3) {\rotatebox{-25}{$\subset$}}; 
\node at (8.3, -0.3) {\rotatebox{25}{$\subset$}};

\node (b-11) at (10, 0) {binary $(-1,1)$-}; 
\node at (11.6, 0) {$\subset$}; 

\node (ra) at (13.4, 0) {Right alternative};


\end{tikzpicture}

\end{center}

To simplify our forthcoming results, we give the classification of associative commutative superalgebras in dimension $3.$

\begin{corollary}
    Let $\rm{A}$ be a complex $3$-dimensional  associative commutative superalgebra of type $(1,2)$. 
    Then $\rm{A}$ is isomorphic to ${\rm R}_{04}^{0},$ 
${\rm R}_{06}^{0},$ 
 ${\rm R}_{17},$ 
 ${\rm R}_{19},$ or
 ${\rm R}_{24}.$

\end{corollary}

\begin{corollary}
    Let $\bf{A}$ be a complex $3$-dimensional  associative commutative superalgebra of type $(2,1)$. 
    Then $\bf{A}$ is isomorphic to     ${\bf R}_{01},$ ${\bf R}_{02},$  
    ${\bf R}_{09},$  
    ${\bf R}_{15},$ 
    ${\bf R}_{17},$ 
     ${\bf R}_{21},$ or  
     ${\bf R}_{37}.$ 

\end{corollary}

\subsubsection{$\mathfrak{perm}$ and binary  $\mathfrak{perm}$ superalgebras}

\begin{definition}
An associative superalgebra is called a $\mathfrak{perm}$  superalgebra if the following identity holds
\begin{longtable}{lclcl}
$x(yz) $&$= $&$(-1)^{|y||z|}x(zy).$
\end{longtable}
\end{definition}

\begin{definition}[see \cite{KS}]
A  superalgebra is called a binary $\mathfrak{perm}$ superalgebra if the following identities hold
\begin{longtable}{c}
$\left( x,y,z\right) \ = \ -\left( -1\right) ^{\left\vert y\right\vert \left\vert
z\right\vert }\left( x,z,y\right) \ = \ 
-\left( -1\right) ^{\left\vert x\right\vert \left\vert y \right\vert }\left( y,x,z\right),$\\
 $(xy)z + \left( -1\right)^{|x||y| +|x||z|+|y||z|}(zy)x \ = \   \left( -1\right)^{|y||z|}(xz)y + \left( -1\right)^{ |x||z|+|y||z|}(zx)y.$ 
\end{longtable}
\end{definition}

\begin{corollary}
Let $\rm{P}$ be a complex $3$-dimensional  $\mathfrak{perm}$ superalgebra of type $(1,2)$. Then $\rm{P}$ is an associative commutative superalgebra,  or it is isomorphic to 

\begin{center}
${\rm R}_{04}^{\alpha\neq 0},$ ${\rm R}_{05},$ ${\rm R}_{06}^{\alpha\neq 0},$    ${\rm R}_{09},$ ${\rm R}_{15},$    ${\rm R}_{22},$   ${\rm R}_{25},$ ${\rm R}_{26},$ or ${\rm R}_{27}.$
    
\end{center}

\end{corollary}

\begin{corollary}
Let $\rm{BP}$ be a complex $3$-dimensional binary  $\mathfrak{perm}$ superalgebra of type $(1,2)$. Then $\rm{BP}$ is a $\mathfrak{perm}$ superalgebra  or isomorphic to 
${\rm R}_{07}^{\alpha}$ or ${\rm R}_{28}.$

    

\end{corollary}

\begin{corollary}
Let $\bf P$  be a complex $3$-dimensional  $\mathfrak{perm}$ superalgebra of type $(2,1)$.   
Then $\bf{P}$ is an associative  commutative superalgebra,  or it is isomorphic to 

\begin{center}
  ${\bf R}_{04},$   ${\bf R}_{10},$ ${\bf R}_{11},$    ${\bf R}_{20},$   ${\bf R}_{25},$ ${\bf R}_{31},$ ${\bf R}_{33},$   ${\bf R}_{38},$ or ${\bf R}_{39}.$
\end{center}

\end{corollary}

\begin{corollary}
Let  $\bf{BP}$ be a complex $3$-dimensional binary  $\mathfrak{perm}$ superalgebra of type $(2,1)$. Then  $\bf{BP}$ is a $\mathfrak{perm}$ superalgebra.

    

\end{corollary}

\subsubsection{Associative and binary associative superalgebras}

\begin{definition}[see \cite{KS}]
A  superalgebra is called a binary associative (=alternative) superalgebra if the following identities hold
\begin{longtable}{c}
$\left( x,y,z\right) \ = \ -\left( -1\right) ^{\left\vert y\right\vert \left\vert
z\right\vert }\left( x,z,y\right) \ = \ 
-\left( -1\right) ^{\left\vert x\right\vert \left\vert y \right\vert }\left( y,x,z\right).$\\
\end{longtable}
\end{definition}

\begin{corollary}
Let $\rm{A}$ be a complex $3$-dimensional associative superalgebra of type $(1,2)$. Then 
$\rm{A}$ is a $\mathfrak{perm}$ superalgebra or it is isomorphic to ${\rm R}_{07}^{0},$
  ${\rm R}_{11},$ ${\rm R}_{13},$ 
  ${\rm R}_{18},$  ${\rm R}_{20},$ or ${\rm R}_{23}.$ 

\end{corollary}

\begin{corollary}
Let $\rm{BA}$ be a complex $3$-dimensional binary associative superalgebra of type $(1,2)$. Then 
it is an associative superalgebra or it is isomorphic to ${\rm R}_{01}^{\pm \bf i\sqrt{3}},$ 
 ${\rm R}_{07}^{\alpha\neq 0},$ or ${\rm R}_{28}.$

\end{corollary}

\begin{corollary}
Let $\bf{A}$ be a complex $3$-dimensional associative superalgebra of type $(2,1)$. Then it is a $\mathfrak{perm}$ superalgebra or it is isomorphic to 
\begin{center}
    ${\bf R}_{03},$  ${\bf R}_{05},$   ${\bf R}_{07},$   
${\bf R}_{13},$ ${\bf R}_{16},$ ${\bf R}_{18},$  ${\bf R}_{22},$ ${\bf R}_{23},$  ${\bf R}_{26},$ ${\bf R}_{29},$  ${\bf R}_{32},$ or  ${\bf R}_{36}.$ 
\end{center}

\end{corollary}

\begin{corollary}
Let $\bf{BA}$ be a complex $3$-dimensional binary associative superalgebra of type $(2,1)$. Then 
it is an associative  superalgebra or it is isomorphic to   ${\bf R}_{28}$ or  ${\bf R}_{30}.$

\end{corollary}

\subsubsection{$\big(-1,1\big)$- and binary $\big(-1,1\big)$-superalgebras}

\begin{definition}
A right-alternative superalgebra is called a $\big(-1,1\big)$-superalgebra if it is Lie-admissible, i.e., the following identity holds
\begin{longtable}{lcl}
$(x,y,z)+(-1)^{|x||y|+|x||z|}(y,z,x)+(-1)^{|z||y|+|x||z|}(z,x,y) $&$=$&$0.$
\end{longtable}
\end{definition}

\begin{definition}[see \cite{ser76bi}]
A right-alternative superalgebra is called a binary $\big(-1,1\big)$-superalgebra if it    satisfies 
 \begin{longtable}{lll}
$(wz,y,x)+(w,z,yx)+(-1)^{|x||y|+|x||z|+|y||z|} \big((wx,y,z)+(w,x,yz) \big)\ +$\\
\multicolumn{1}{r}{ $(-1)^{\big(|x|+|y|\big)\big(|z|+|w|\big) + |x||y|+|z||w|} (x,yz,w)+(-1)^{|x| \big(|y|+|z|+|w|\big) } (x,wz,y) $}&$+$\\
$(-1)^{|w||y|+|w||z|+|y||z|} \big( (yz,w,x)+(y,z,wx)\big)+(-1)^{\big(|x|+|y|\big)\big(|z|+|w|\big)} \big((yx,w,z)+(y,x,wz) \big)$&$+$\\
\multicolumn{1}{r}{ $(-1)^{|w| \big(|x|+|y|+|z|\big) }(z,yx,w)+
(-1)^{|x| |y|+|z||w| }(z,wx,y)$}&$=$&$0.$
\end{longtable}
\end{definition}

\begin{corollary}
Let $\rm{M}$ be a complex $3$-dimensional $\big(-1,1\big)$-superalgebra of type $(1,2)$. Then $\rm{M}$ is associative, or it is isomorphic to 
${\rm R}_{01}^{3},$ 
${\rm R}_{10},$ or ${\rm R}_{14}.$

\end{corollary}

\begin{corollary}
Let $\rm{BM}$ be a complex $3$-dimensional binary $\big(-1,1\big)$-superalgebra of type $(1,2)$. Then 
$\rm{BM}$ is a $\big(-1,1\big)$-superalgebra or it is isomorphic to 
${\rm R}_{01}^{\alpha \notin \{ -1; 3\}},$ ${\rm R}_{02},$ ${\rm R}_{03},$ 
${\rm R}_{08},$ 
${\rm R}_{12},$ or ${\rm R}_{28}.$

\end{corollary}

\begin{corollary}
Let $\rm{R}$ be a complex $3$-dimensional right alternative superalgebra of type $(1,2)$. Then 
$\rm{R}$ is a binary $\big(-1,1\big)$-superalgebra or it is isomorphic to 
${\rm R}_{07}^{\alpha \neq -1},$ 
${\rm R}_{16},$ or 
${\rm R}_{21}.$ 

\end{corollary}

\begin{corollary}
Let $\bf{M}$ be a complex $3$-dimensional $\big(-1,1\big)$-superalgebra of type $(2,1)$. Then 
$\bf{M}$ is associative.

\end{corollary}

\begin{corollary}
Let $\bf{BM}$ be a complex $3$-dimensional binary $\big(-1,1\big)$-superalgebra of type $(2,1)$. Then 
$\bf{BM}$ is binary associative, or it is isomorphic to 
${\bf R}_{06},$ 
${\bf R}_{08},$
${\bf R}_{12},$
${\bf R}_{14},$ or
${\bf R}_{19}.$ 

     

\end{corollary}

\begin{corollary}
Let $\bf{R}$ be a complex $3$-dimensional right alternative superalgebra of type $(2,1)$. Then 
$\bf{R}$ is a binary $\big(-1,1\big)$-superalgebra or it is isomorphic to 
${\bf R}_{24},$ 
${\bf R}_{27},$ 
${\bf R}_{34},$ or 
${\bf R}_{35}.$ 
\end{corollary}

\section{The geometric classification of superalgebras}

\subsection{Preliminaries: degenerations and geometric classification}

Let \( V = V_0 \oplus V_1 \) be a \( \mathbb{Z}_2 \)-graded vector space with a fixed homogeneous basis  
\( \big\{e_1, \ldots, e_m, f_1, \ldots, f_n \big\}\). 
A  superalgebra structure on \(V\) can be described via structure constants 
\((\alpha_{ij}^k, \beta_{ij}^k, \gamma_{ij}^k, \delta_{ij}^k) \in \mathbb{C}^{m^3+3mn^2}\), where the multiplication is defined as:
\[
e_i e_j = \sum_{k=1}^{m} \alpha_{ij}^k e_k, \quad 
e_i f_j = \sum_{k=1}^{n} \beta_{ij}^k f_k, \quad 
f_i e_j = \sum_{k=1}^{n} \gamma_{ij}^k f_k, \quad 
f_i f_j = \sum_{k=1}^{m} \delta_{ij}^k e_k.
\]
Let $\mathcal{S}^{m,n}$ denote the set of all superalgebras of dimension $(m,n)$ defined by a family of polynomial superidentities $T$, regarded as a subset $\mathbb{L}(T)$ of the affine variety $\operatorname{Hom}(V \otimes V, V)$. Then $\mathcal{S}^{m,n}$ is a Zariski-closed subset of the variety $\operatorname{Hom}(V \otimes V, V)$. 

The group $G = (\operatorname{Aut} V)_0 \simeq \operatorname{GL}(V_0) \oplus \operatorname{GL}(V_1)$ acts on $\mathcal{S}^{m,n}$ by conjugation:
\[
(g * \mu)(x \otimes y) = g \mu(g^{-1} x \otimes g^{-1} y),
\]
for all $x, y \in V$, $\mu \in \mathbb{L}(T)$, and $g \in G$.

Let $\mathcal{O}(\mu)$ denote the orbit of $\mu \in \mathbb{L}(T)$ under the action of $G$, and let $\overline{\mathcal{O}(\mu)}$ be the Zariski closure of $\mathcal{O}(\mu)$. Suppose $J, J' \in \mathcal{S}^{m,n}$ are represented by $\lambda, \mu \in \mathbb{L}(T)$, respectively. We say that $\lambda$ degenerates to $\mu$, denoted $\lambda \to \mu$, if $\mu \in \overline{\mathcal{O}(\lambda)}$. In this case, we have $\overline{\mathcal{O}(\mu)} \subset \overline{\mathcal{O}(\lambda)}$. Therefore, the notion of degeneration does not depend on the particular representatives, and we write $J \to J'$ instead of $\lambda \to \mu$, and $\mathcal{O}(J)$ instead of $\mathcal{O}(\lambda)$. 
We write $J \not\to J'$ to indicate that $J' \notin \overline{\mathcal{O}(J)}$.

If $J$ is represented by $\lambda \in \mathbb{L}(T)$, we say that $J$ is \textit{rigid} in $\mathbb{L}(T)$ if $\mathcal{O}(\lambda)$ is an open subset of $\mathbb{L}(T)$. A subset of a variety is called \textit{irreducible} if it cannot be written as a union of two proper closed subsets. A maximal irreducible closed subset is called an \textit{irreducible component}. In particular, $J$ is rigid in $\mathcal{S}^{m,n}$ if and only if $\overline{\mathcal{O}(\lambda)}$ is an irreducible component of $\mathbb{L}(T)$. It is a well-known fact that every affine variety admits a unique decomposition into finitely many irreducible components.


To find degenerations, we use the standard methods, described in \cite{GRH,als,ahk,BC99} and so on.
To prove a non-degeneration $J \not\to J',$ we use the standard argument from Lemma, whose proof is the same as the proof of   \cite[Lemma 1.5]{GRH}.

\begin{lemma}\label{gmain}
Let $\mathfrak{B}$ be a Borel subgroup of ${\rm GL}(\mathbb V)$ and ${\rm R}\subset \mathbb{L}(T)$ be a $\mathfrak{B}$-stable closed subset.
If $J  \to J'$ and  the superalgebra $J $ can be represented by a structure $\mu\in{\rm R}$, then there is $\lambda\in {\rm R}$ representing $J'$.
\end{lemma}

\subsection{The geometric classification}

\subsubsection{Associative commutative   superalgebras}

\begin{theorem}[see, \cite{ahk}]\label{asscomg}
The following statements are true:

\begin{enumerate}
    \item[$({\rm 1})$]
    The variety of associative commutative  superalgebras of type $(1,2)$ 
    has dimension $3,$ 
    four rigid superalgebras and 
    four irreducible components given by
    \begin{center}
$\mathcal{C}_1=\overline{\mathcal{O}({\rm R}_{04}^0)},$  
$\mathcal{C}_2=\overline{\mathcal{O}({\rm R}_{17})},$  
$\mathcal{C}_3=\overline{\mathcal{O}({\rm R}_{19})},$ and 
$\mathcal{C}_4=\overline{\mathcal{O}({\rm R}_{24})}.$  
    \end{center}
    \item[$({\rm 2})$]
    The variety of associative commutative  superalgebras of type $(2,1)$ 
    has dimension $4,$ 
    two  rigid superalgebras and 
    two irreducible components given by
      \begin{center}  
$\mathcal{C}_1=\overline{\mathcal{O}({\bf R}_{01})}$  and
$\mathcal{C}_2=\overline{\mathcal{O}({\bf R}_{02})}.$   
    \end{center}

\end{enumerate}
    
\end{theorem}

\subsubsection{$\mathfrak{perm}$  superalgebras}

\begin{theorem}\label{pgeo1}
The variety of complex $3$-dimensional $\mathfrak{perm}$ superalgebras of type $(1,2)$  has 
dimension  $4$   and it has  $8$  irreducible components defined by  
\begin{center}
$\mathcal{C}_1=\overline{\mathcal{O}( {\rm R}_{04}^\alpha)},$ \
$\mathcal{C}_2=\overline{\mathcal{O}( {\rm R}_{06}^\alpha)},$ \
$\mathcal{C}_3=\overline{\mathcal{O}( {\rm R}_{09})},$ \
$\mathcal{C}_4=\overline{\mathcal{O}( {\rm N}_{15})},$ \\
$\mathcal{C}_5=\overline{\mathcal{O}( {\rm R}_{17})},$ \ 
$\mathcal{C}_6=\overline{\mathcal{O}( {\rm R}_{19})},$ \
$\mathcal{C}_7=\overline{\mathcal{O}( {\rm R}_{22})},$  and
$\mathcal{C}_8=\overline{\mathcal{O}( {\rm R}_{24})}.$ \

 \end{center}
In particular, there are only $6$ rigid superalgebras in this variety.
 
\end{theorem}

\begin{proof}
Taking results from Theorem \ref{asscomg},  we have to mention only four new degenerations

\begin{center}
${\rm R}_{04}^{-\frac{t}{t+2}}  {\xrightarrow{ \big(    
\frac{2 t}{t+2}e_1,\ 
tf_1 ,\  
f_1+f_2 \big)}}  {\rm R}_{05};$ \ 
${\rm R}_{04}^{-\frac{t+2}{t}}  {\xrightarrow{ \big(    
-2 (t+2)e_1,\ 
(t+2)f_1,\  
tf_1+tf_2 \big)}} {\rm R}_{25};$ \\
${\rm R}_{04}^{-1}  {\xrightarrow{ \big(    
-2e_1,\ 
f_1+f_2,\  
tf_2 \big)}}  {\rm R}_{26};$  \
${\rm R}_{06}^{\frac{2}{t}-1}  {\xrightarrow{ \big(    
\frac{t}{2}e_1,\ 
f_1,\  
 f_2 \big)}}  {\rm R}_{27}.$ 
\end{center}
\noindent
All non-degenerations follow from the following observations.
\begin{enumerate}
    \item[$\bullet$] 
$ {\rm R}_{04}^{\alpha}   \not\to 
\big\{ {\rm R}_{06}^{\alpha},  {\rm R}_{09},  {\rm R}_{15},  {\rm R}_{17},  {\rm R}_{19},  {\rm R}_{22},  {\rm R}_{24} \big\},$ since 
$\big(({\rm R}_{04}^{\alpha})_0\big)^2= ({\rm R}_{04}^{\alpha})_0({\rm R}_{04}^{\alpha})_1= 0.$

  \item[$\bullet$] 
$ {\rm R}_{06}^{\alpha}   \not\to 
 \big\{  {\rm R}_{19},  {\rm R}_{22},  {\rm R}_{24} \big\},$ since 
$\big(({\rm R}_{06}^{\alpha})_0\big)^2= 0.$

\item[$\bullet$]
$ {\rm R}_{09}   \not\to 
\big\{  {\rm R}_{19},  {\rm R}_{22},  {\rm R}_{24} \big\},$ since  
${\rm R}_{09}$ satisfies $\big\{ c_{11}^1=c_{21}^2, c_{12}^2=c_{13}^3=c_{31}^3=0 \big\}.$ 

\item[$\bullet$]
$ {\rm R}_{15}   \not\to 
\big\{  {\rm R}_{19},  {\rm R}_{22},  {\rm R}_{24} \big\},$ since  
${\rm R}_{15}$ satisfies $\big\{ c_{11}^1=c_{21}^2=c_{31}^3, 
c_{12}^2=c_{13}^3 =0 \big\}.$ 
\end{enumerate}

\noindent The dimensions of orbit closures are given below:
\begin{longtable}{lclclclclclclclclcl}
&&&&&&${\rm dim} \ \overline{\mathcal{O}( {\rm R}_{04}^{\alpha})} $& $=$&$ 
 4;$\\ 

${\rm dim} \ \overline{\mathcal{O}( {\rm R}_{06}^{\alpha})} $&$ = $&
${\rm dim} \ \overline{\mathcal{O}( {\rm R}_{09})}  $&$ = $&
${\rm dim} \ \overline{\mathcal{O}( {\rm R}_{15})}  $&$ = $&
${\rm dim} \ \overline{\mathcal{O}( {\rm R}_{17})}  $&$ = $& $3;$\\

&&${\rm dim} \ \overline{\mathcal{O}( {\rm R}_{19})}  $&$ = $&
${\rm dim} \ \overline{\mathcal{O}( {\rm R}_{22})}  $&$ = $&
${\rm dim} \ \overline{\mathcal{O}( {\rm R}_{24})}  $&$ = $& $1.$\\
\end{longtable}

\end{proof}

\begin{theorem}\label{pgeo2}
The variety of complex $3$-dimensional $\mathfrak{perm}$ superalgebras of type $(2,1)$  has 
dimension  $4$   and it has  $7$  irreducible components defined by  
\begin{center}
$\mathcal{C}_1=\overline{\mathcal{O}( {\bf R}_{01})},$ \
$\mathcal{C}_2=\overline{\mathcal{O}( {\bf R}_{02})},$ \
$\mathcal{C}_3=\overline{\mathcal{O}( {\bf R}_{04})},$ \
$\mathcal{C}_4=\overline{\mathcal{O}( {\bf R}_{10})},$ \\
$\mathcal{C}_5=\overline{\mathcal{O}( {\bf R}_{25})},$ \
$\mathcal{C}_6=\overline{\mathcal{O}( {\bf R}_{31})},$ \ and 
$\mathcal{C}_7=\overline{\mathcal{O}( {\bf R}_{33})}.$ \
 
 \end{center}
In particular, there are only $7$ rigid superalgebras in this variety.
 
\end{theorem}

\begin{proof}
Taking results from Theorem \ref{asscomg},  we have to mention only four new degenerations
\begin{center} 
${\bf R}_{04}
{\xrightarrow{ \big(    
e_1,\ 
te_2 ,\  
f_1 \big)}}  
{\bf R}_{11};$ \  
${\bf R}_{04}
{\xrightarrow{ \big(    
e_1+(t^2+1)e_2,\ 
te_2 ,\  
f_1 \big)}}  
{\bf R}_{20};$ \\   
${\bf R}_{10}
{\xrightarrow{ \big(    
te_1 -te_2,\ 
 t^2e_2,\  
tf_1 \big)}}  
{\bf R}_{38};$ \  
${\bf R}_{10}
{\xrightarrow{ \big(    
e_2,\ 
te_1 ,\  
f_1 \big)}}  
{\bf R}_{39}.$  
\end{center}
\noindent
All non-degenerations follow from the following observations.
\begin{enumerate}
    \item[$\bullet$] 
$ \big\{{\bf R}_{04}, {\bf R}_{10} \big\}   \not\to 
\big\{ {\bf R}_{25},  {\bf R}_{31},  {\bf R}_{33} \big\},$ since 
$({\bf R}_{04})_0$ and $({\bf R}_{10})_0$ are commmutative.
 
\end{enumerate}

\noindent The dimensions of orbit closures are given below:
\begin{longtable}{lclclclclclclclclcl}
${\rm dim} \ \overline{\mathcal{O}( {\bf R}_{01})} $& $=$& 
${\rm dim} \ \overline{\mathcal{O}( {\bf R}_{02})} $& $=$& 
${\rm dim} \ \overline{\mathcal{O}( {\bf R}_{04})} $& $=$& 
${\rm dim} \ \overline{\mathcal{O}( {\bf R}_{10})} $& $=$& 
 $4;$\\ 
&&${\rm dim} \ \overline{\mathcal{O}( {\bf R}_{25})} $& $=$& 
${\rm dim} \ \overline{\mathcal{O}( {\bf R}_{31})} $& $=$& 
${\rm dim} \ \overline{\mathcal{O}( {\bf R}_{33})} $& $=$& 
 $2.$\\ 

\end{longtable}

\end{proof}

\subsubsection{Binary  $\mathfrak{perm}$   superalgebras}

\begin{theorem}\label{bpgeo1}
The variety of complex $3$-dimensional binary $\mathfrak{perm}$ superalgebras of type $(1,2)$  has 
dimension  $4$   and it has  $8$  irreducible components defined by  
\begin{center}
$\mathcal{C}_1=\overline{\mathcal{O}( {\rm R}_{04}^\alpha)},$ \
$\mathcal{C}_2=\overline{\mathcal{O}( {\rm R}_{07}^\alpha)},$ \
$\mathcal{C}_3=\overline{\mathcal{O}( {\rm R}_{09})},$ \
$\mathcal{C}_4=\overline{\mathcal{O}( {\rm N}_{15})},$ \\
$\mathcal{C}_5=\overline{\mathcal{O}( {\rm R}_{17})},$ \ 
$\mathcal{C}_6=\overline{\mathcal{O}( {\rm R}_{19})},$ \
$\mathcal{C}_7=\overline{\mathcal{O}( {\rm R}_{22})},$ and 
$\mathcal{C}_8=\overline{\mathcal{O}( {\rm R}_{24})}.$

 \end{center}
In particular, there are only $6$ rigid superalgebras in this variety.
 
\end{theorem}

\begin{proof}
Taking results from Theorem \ref{pgeo1},  we have to mention only two new degenerations
\begin{center}
${\rm R}_{07}^{\alpha} \ {\xrightarrow{ \big(    
e_1,\ 
t f_1,\  
tf_2 \big)}}\  {\rm R}_{06}^\alpha;$ \ 
${\rm R}_{07}^{\frac{2}{t}-1} \ {\xrightarrow{ \big(    
e_1,\ 
t f_1,\  
\frac{2}{t}f_2 \big)}}\  {\rm R}_{28}.$ \\

\end{center}
\noindent
All non-degenerations follow from 
${\rm dim} \ \overline{\mathcal{O}( {\rm R}_{07}^{\alpha})} \ = 4$ and 
the following observations.
\begin{enumerate}
    \item[$\bullet$] 
$ {\rm R}_{07}^{\alpha}   \not\to 
\big\{   {\rm R}_{09},  {\rm R}_{15},  {\rm R}_{17},  {\rm R}_{19},  {\rm R}_{22},  {\rm R}_{24} \big\},$ since 
$\big(({\rm R}_{07}^{\alpha})_0\big)^2= 0.$
 
\end{enumerate}

\end{proof}

\begin{theorem}\label{bpgeo2}
The variety of complex $3$-dimensional binary $\mathfrak{perm}$ superalgebras of type $(2,1)$  has 
dimension  $4$   and it has  $7$  irreducible components defined by  
\begin{center}
$\mathcal{C}_1=\overline{\mathcal{O}( {\bf R}_{01})},$ \
$\mathcal{C}_2=\overline{\mathcal{O}( {\bf R}_{02})},$ \
$\mathcal{C}_3=\overline{\mathcal{O}( {\bf R}_{04})},$ \
$\mathcal{C}_4=\overline{\mathcal{O}( {\bf R}_{10})},$ \\
$\mathcal{C}_5=\overline{\mathcal{O}( {\bf R}_{25})},$ \
$\mathcal{C}_6=\overline{\mathcal{O}( {\bf R}_{31})},$ \ and 
$\mathcal{C}_7=\overline{\mathcal{O}( {\bf R}_{33})}.$ \
 
 \end{center}
In particular, there are only $7$ rigid superalgebras in this variety.
 
\end{theorem}

\subsubsection{Associative   superalgebras}

\begin{theorem}\label{assgeo1}
The variety of complex $3$-dimensional associative  superalgebras of type $(1,2)$  has 
dimension  $4$   and it has  $12$  irreducible components defined by  
\begin{center}
$\mathcal{C}_1=\overline{\mathcal{O}( {\rm R}_{04}^\alpha)},$ \
$\mathcal{C}_2=\overline{\mathcal{O}( {\rm R}_{06}^\alpha)},$ \
$\mathcal{C}_3=\overline{\mathcal{O}( {\rm R}_{09})},$ \
$\mathcal{C}_4=\overline{\mathcal{O}( {\rm R}_{11})},$ \\
$\mathcal{C}_5=\overline{\mathcal{O}( {\rm R}_{13})},$ \
$\mathcal{C}_6=\overline{\mathcal{O}( {\rm R}_{15})},$ \
$\mathcal{C}_7=\overline{\mathcal{O}( {\rm R}_{18})},$ \ 
$\mathcal{C}_8=\overline{\mathcal{O}( {\rm R}_{19})},$ \\
$\mathcal{C}_9=\overline{\mathcal{O}( {\rm R}_{20})},$ \
$\mathcal{C}_{10}=\overline{\mathcal{O}( {\rm R}_{22})},$  \ 
$\mathcal{C}_{11}=\overline{\mathcal{O}( {\rm R}_{23})},$  and
$\mathcal{C}_{12}=\overline{\mathcal{O}( {\rm R}_{24})}.$ \

 \end{center}
In particular, there are only $10$ rigid superalgebras in this variety.
 
\end{theorem}

\begin{proof}
Taking results from Theorem \ref{pgeo1},  we have to mention only one  two degenerations
\begin{center}
${\rm R}_{18} \ {\xrightarrow{ \big(    
t^2e_1,\ 
tf_1+f_2,\  
t^2 (1+t)f_1+tf_2 \big)}} \ {\rm R}_{07}^0;$ \  
${\rm R}_{18} \ {\xrightarrow{ \big(    
e_1,\ 
tf_1 ,\  
f_2 \big)}} \ {\rm R}_{17}.$ 
\end{center}
 
\noindent
All non-degenerations follow from the following observations.
\begin{enumerate}
    \item[$\bullet$] 
$ {\rm R}_{18}    \not\to 
\left\{ 
\begin{array}{l}
{\rm R}_{06}^{\alpha},  {\rm R}_{09}, {\rm R}_{11}, {\rm R}_{13},
{\rm R}_{15},  \\  
{\rm R}_{19}, {\rm R}_{20}, {\rm R}_{22}, {\rm R}_{23},  {\rm R}_{24} 
\end{array}\right\},$ since 
 ${\rm R}_{18}$   satisfies 
       $\big\{ c_{11}^1=c_{12}^2=c_{21}^2, c_{13}^3=c_{31}^3=0 \big\}.$

  \item[$\bullet$] 
$ {\rm R}_{11}    \not\to 
  \big\{ {\rm R}_{19},   {\rm R}_{22}, 
{\rm R}_{23},  {\rm R}_{24} \big\},$ since 
 ${\rm R}_{11}$   satisfies  
 $\big\{ c_{11}^1=c_{12}^2, c_{21}^1=c_{13}^3=c_{31}^3=0 \big\}.$

   \item[$\bullet$] 
$ {\rm R}_{13}    \not\to 
  \big\{ {\rm R}_{19},   {\rm R}_{22}, 
{\rm R}_{23},  {\rm R}_{24} \big\},$ since 
 ${\rm R}_{13}$   satisfies  
 $\big\{ c_{11}^1=c_{12}^2=c_{13}^3, c_{21}^1=c_{31}^3=0 \big\}.$

  \item[$\bullet$] 
$ {\rm R}_{20}    \not\to 
  \big\{ {\rm R}_{19},   {\rm R}_{22}, 
{\rm R}_{23},  {\rm R}_{24} \big\},$ since 
 ${\rm R}_{20}$   satisfies  
 $\big\{ c_{11}^1=c_{12}^2=c_{31}^3, c_{21}^1=c_{13}^3=0 \big\}.$

\end{enumerate}

\noindent The dimensions of orbit closures are given below:
\begin{longtable}{lclclclclclclclclcl}
&&&&${\rm dim} \ \overline{\mathcal{O}( {\rm R}_{04}^{\alpha})} $& $=$&
${\rm dim} \ \overline{\mathcal{O}( {\rm R}_{18})}  $&$ = $&$ 
 4;$\\ 

&&${\rm dim} \ \overline{\mathcal{O}( {\rm R}_{06}^{\alpha})} $&$ = $&
${\rm dim} \ \overline{\mathcal{O}( {\rm R}_{09})}  $&$ = $&
${\rm dim} \ \overline{\mathcal{O}( {\rm R}_{11})}  $&$ = $&\\&&
${\rm dim} \ \overline{\mathcal{O}( {\rm R}_{13})}  $&$ = $&
${\rm dim} \ \overline{\mathcal{O}( {\rm R}_{15})}  $&$ = $&
${\rm dim} \ \overline{\mathcal{O}( {\rm R}_{20})}  $&$ = $& $3;$\\

${\rm dim} \ \overline{\mathcal{O}( {\rm R}_{19})}  $&$ = $&
${\rm dim} \ \overline{\mathcal{O}( {\rm R}_{22})}  $&$ = $&
${\rm dim} \ \overline{\mathcal{O}( {\rm R}_{23})}  $&$ = $&
${\rm dim} \ \overline{\mathcal{O}( {\rm R}_{24})}  $&$ = $& $1.$\\
\end{longtable}

\end{proof}

\begin{theorem}\label{assgeo2}
The variety of complex $3$-dimensional associative superalgebras of type $(2,1)$  has 
dimension  $5$   and it has  $13$  irreducible components defined by  
\begin{center}
$\mathcal{C}_1=\overline{\mathcal{O}( {\bf R}_{01})},$ \
$\mathcal{C}_2=\overline{\mathcal{O}( {\bf R}_{03})},$ \
$\mathcal{C}_3=\overline{\mathcal{O}( {\bf R}_{04})},$ \
$\mathcal{C}_4=\overline{\mathcal{O}( {\bf R}_{05})},$ \
$\mathcal{C}_5=\overline{\mathcal{O}( {\bf R}_{07})},$ \\
$\mathcal{C}_6=\overline{\mathcal{O}( {\bf R}_{23})},$ \
$\mathcal{C}_7=\overline{\mathcal{O}( {\bf R}_{25})},$ \
$\mathcal{C}_8=\overline{\mathcal{O}( {\bf R}_{26})},$ \
$\mathcal{C}_9=\overline{\mathcal{O}( {\bf R}_{29})},$ \\
$\mathcal{C}_{10}=\overline{\mathcal{O}( {\bf R}_{31})},$ \
$\mathcal{C}_{11}=\overline{\mathcal{O}( {\bf R}_{32})},$ \ 
$\mathcal{C}_{12}=\overline{\mathcal{O}( {\bf R}_{33})},$ \ and 
$\mathcal{C}_{13}=\overline{\mathcal{O}( {\bf R}_{36})}.$ \\
 
 \end{center}
In particular, there are only $13$ rigid superalgebras in this variety.
 
\end{theorem}

\begin{proof}
Taking results from Theorem \ref{pgeo2},  we have to mention only six new degenerations
\begin{center} 
${\bf R}_{03}
{\xrightarrow{ \big(    
e_1,\ 
e_2 ,\  
t f_1 \big)}}  
{\bf R}_{02};$ \   
${\bf R}_{03}
{\xrightarrow{ \big(    
t^4e_1+ e_2,\ 
t^2e_1,\  
tf_1 \big)}}  
{\bf R}_{10};$ \   
${\bf R}_{05}
{\xrightarrow{ \big(    
e_1,\ 
t e_2,\  
f_1 \big)}}  
{\bf R}_{13};$ \\  
${\bf R}_{03}
{\xrightarrow{ \big(    
e_1,\ 
t e_2,\  
f_1 \big)}}  
{\bf R}_{16};$ \  
${\bf R}_{05}
{\xrightarrow{ \big(    
e_1+e_2,\ 
t e_2,\  
f_1 \big)}}  
{\bf R}_{18};$ \  
${\bf R}_{03}
{\xrightarrow{ \big(    
e_1+e_2,\ 
-t^2 e_2,\  
t f_1 \big)}}  
{\bf R}_{22}.$ \  
\end{center}
\noindent
All non-degenerations follow from the following observations.
\begin{enumerate}
    \item[$\bullet$] 
${\bf R}_{03}   \not\to 
\big\{ 
{\bf R}_{01},
{\bf R}_{04},
{\bf R}_{05},
{\bf R}_{07},
{\bf R}_{23},
{\bf R}_{25}, 
{\bf R}_{26}, 
{\bf R}_{29}, {\bf R}_{31}, {\bf R}_{32}, {\bf R}_{33}, {\bf R}_{36} \big\},$ since 
$({\bf R}_{03})_0$ is commutative  and ${\bf R}_{03}$ satisfies 
$\{ c_{11}^1=c_{13}^3=c_{31}^3\}.$

     \item[$\bullet$] 
$\big\{{\bf R}_{05}, {\bf R}_{07} \big\}  \not\to 
\big\{ 
{\bf R}_{23},
{\bf R}_{25}, 
{\bf R}_{26}, 
{\bf R}_{29}, {\bf R}_{31}, {\bf R}_{32}, {\bf R}_{33}, {\bf R}_{36} \big\},$ since 
$({\bf R}_{05})_0$ and  $({\bf R}_{07})_0$ are commutative.

\end{enumerate}

\noindent The dimensions of orbit closures are given below:
\begin{longtable}{lclclclclclclclclcl}
&&&&&&
${\rm dim} \ \overline{\mathcal{O}( {\bf R}_{03})} $& $=$& 
 $5;$\\ 
${\rm dim} \ \overline{\mathcal{O}( {\bf R}_{01})} $& $=$& 
${\rm dim} \ \overline{\mathcal{O}( {\bf R}_{04})} $& $=$& 
${\rm dim} \ \overline{\mathcal{O}( {\bf R}_{05})} $& $=$& 
${\rm dim} \ \overline{\mathcal{O}( {\bf R}_{07})} $& $=$& 
 $4;$\\ 

${\rm dim} \ \overline{\mathcal{O}( {\bf R}_{23})} $& $=$& 
${\rm dim} \ \overline{\mathcal{O}( {\bf R}_{25})} $& $=$& 
${\rm dim} \ \overline{\mathcal{O}( {\bf R}_{26})} $& $=$& 
${\rm dim} \ \overline{\mathcal{O}( {\bf R}_{29})} $& $=$\\ 
${\rm dim} \ \overline{\mathcal{O}( {\bf R}_{31})} $& $=$& 
${\rm dim} \ \overline{\mathcal{O}( {\bf R}_{32})} $& $=$& 
${\rm dim} \ \overline{\mathcal{O}( {\bf R}_{33})} $& $=$& 
${\rm dim} \ \overline{\mathcal{O}( {\bf R}_{36})} $& $=$& 
 $2.$\\ 

\end{longtable}

\end{proof}

\subsubsection{Binary associative     superalgebras}

\begin{theorem}\label{altgeo1}
The variety of complex $3$-dimensional binary associative (alternative)  superalgebras of type $(1,2)$  has 
dimension  $4$   and it has  $13$  irreducible components defined by  
\begin{center}
$\mathcal{C}_1=\overline{\mathcal{O}( {\rm R}_{01}^{{\bf i}\sqrt{3}})},$ \
$\mathcal{C}_2=\overline{\mathcal{O}( {\rm R}_{01}^{-{\bf i}\sqrt{3}})},$ \
$\mathcal{C}_3=\overline{\mathcal{O}( {\rm R}_{04}^\alpha)},$ \
$\mathcal{C}_4=\overline{\mathcal{O}( {\rm R}_{07}^\alpha)},$ \
$\mathcal{C}_5=\overline{\mathcal{O}( {\rm R}_{09})},$ \\
$\mathcal{C}_6=\overline{\mathcal{O}( {\rm R}_{11})},$ \
$\mathcal{C}_7=\overline{\mathcal{O}( {\rm R}_{13})},$ \
$\mathcal{C}_8=\overline{\mathcal{O}( {\rm R}_{15})},$ \
$\mathcal{C}_9=\overline{\mathcal{O}( {\rm R}_{18})},$ \\ 
$\mathcal{C}_{10}=\overline{\mathcal{O}( {\rm R}_{19})},$ \
$\mathcal{C}_{11}=\overline{\mathcal{O}( {\rm R}_{22})},$  \ 
$\mathcal{C}_{12}=\overline{\mathcal{O}( {\rm R}_{23})},$  and
$\mathcal{C}_{13}=\overline{\mathcal{O}( {\rm R}_{24})}.$ \

 \end{center}
In particular, there are only $11$ rigid superalgebras in this variety.
 
\end{theorem}

\begin{proof}

Taking results from Theorem \ref{pgeo1},  we have to mention only two new degenerations
\begin{center}
${\rm R}_{07}^{\alpha} \ {\xrightarrow{ \big(    
e_1,\ 
t f_1,\  
tf_2 \big)}}\  {\rm R}_{06}^\alpha;$ \ 
${\rm R}_{07}^{\frac{2}{t}-1} \ {\xrightarrow{ \big(    
e_1,\ 
t f_1,\  
\frac{2}{t}f_2 \big)}}\  {\rm R}_{28}.$ \\

\end{center}
\noindent
All non-degenerations follow from 
$\big( ({\rm R}_{01}^{\alpha})_0 \big)^2 = 
\big( ({\rm R}_{07}^{\alpha})_0 \big)^2 =0.$

\end{proof}

\begin{theorem}\label{altgeo2}
The variety of complex $3$-dimensional binary associative superalgebras of type $(2,1)$  has 
dimension  $5$   and it has  $13$  irreducible components defined by  
\begin{center}
$\mathcal{C}_1=\overline{\mathcal{O}( {\bf R}_{01})},$ \
$\mathcal{C}_2=\overline{\mathcal{O}( {\bf R}_{03})},$ \
$\mathcal{C}_3=\overline{\mathcal{O}( {\bf R}_{04})},$ \
$\mathcal{C}_4=\overline{\mathcal{O}( {\bf R}_{05})},$ \
$\mathcal{C}_5=\overline{\mathcal{O}( {\bf R}_{07})},$ \\
$\mathcal{C}_6=\overline{\mathcal{O}( {\bf R}_{23})},$ \
$\mathcal{C}_7=\overline{\mathcal{O}( {\bf R}_{25})},$ \
$\mathcal{C}_8=\overline{\mathcal{O}( {\bf R}_{28})},$ \
$\mathcal{C}_9=\overline{\mathcal{O}( {\bf R}_{30})},$ \\
$\mathcal{C}_{10}=\overline{\mathcal{O}( {\bf R}_{31})},$ \
$\mathcal{C}_{11}=\overline{\mathcal{O}( {\bf R}_{32})},$ \ 
$\mathcal{C}_{12}=\overline{\mathcal{O}( {\bf R}_{33})},$ \ and 
$\mathcal{C}_{13}=\overline{\mathcal{O}( {\bf R}_{36})}.$ \\
 
 \end{center}
In particular, there are only $13$ rigid superalgebras in this variety.
 
\end{theorem}

\begin{proof}
Taking results from Theorem \ref{assgeo2},  we have to mention only six new degenerations
\begin{center} 
${\bf R}_{28}
{\xrightarrow{ \big(    
e_1,\ 
e_2 ,\  
t f_1 \big)}}  
{\bf R}_{26};$ \   
${\bf R}_{30}
{\xrightarrow{ \big(    
e_1,\ 
e_2 ,\  
t f_1 \big)}}  
{\bf R}_{29}.$ \   
\end{center}
\noindent
All non-degenerations follow from the following observations.
\begin{enumerate}
    \item[$\bullet$] 
${\bf R}_{28}   \not\to 
\left\{\begin{array}{l} 
{\bf R}_{23}, {\bf R}_{25},  {\bf R}_{31},\\ 
{\bf R}_{32}, {\bf R}_{33}, {\bf R}_{36}
 \end{array}\right\},$ since 
${\bf R}_{28}$   satisfies 
$\big\{ c_{11}^1=c_{13}^3=c_{21}^2, \ c_{31}^3=c_{12}^2=0\big\}.$  

  \item[$\bullet$] ${\bf R}_{30}   \not\to 
\left\{\begin{array}{l} 
{\bf R}_{23}, {\bf R}_{25},  {\bf R}_{31},\\ 
{\bf R}_{32}, {\bf R}_{33}, {\bf R}_{36}
 \end{array}\right\},$ since 
${\bf R}_{30}$   satisfies 
$\big\{ c_{11}^1=c_{31}^3=c_{12}^2\ c_{13}^3=c_{21}^2=0\big\}.$

\end{enumerate}

\noindent The dimensions of orbit closures are given below:
\begin{longtable}{lclclclclclclclclcl}
&&&&&&
${\rm dim} \ \overline{\mathcal{O}( {\bf R}_{03})} $& $=$& 
 $5;$\\ 
${\rm dim} \ \overline{\mathcal{O}( {\bf R}_{01})} $& $=$& 
${\rm dim} \ \overline{\mathcal{O}( {\bf R}_{04})} $& $=$& 
${\rm dim} \ \overline{\mathcal{O}( {\bf R}_{05})} $& $=$& 
${\rm dim} \ \overline{\mathcal{O}( {\bf R}_{07})} $& $=$& 
 $4;$\\ 

&&&&${\rm dim} \ \overline{\mathcal{O}( {\bf R}_{28})} $& $=$& 
${\rm dim} \ \overline{\mathcal{O}( {\bf R}_{30})} $& $=$&
$3;$\\

&&${\rm dim} \ \overline{\mathcal{O}( {\bf R}_{23})} $& $=$& 
${\rm dim} \ \overline{\mathcal{O}( {\bf R}_{25})} $& $=$& 
${\rm dim} \ \overline{\mathcal{O}( {\bf R}_{31})} $& $=$\\ 
&&
${\rm dim} \ \overline{\mathcal{O}( {\bf R}_{32})} $& $=$& 
${\rm dim} \ \overline{\mathcal{O}( {\bf R}_{33})} $& $=$& 
${\rm dim} \ \overline{\mathcal{O}( {\bf R}_{36})} $& $=$& 
 $2.$\\ 

\end{longtable}

\end{proof}

\subsubsection{$\big(-1,1\big)$-superalgebras}

\begin{theorem}\label{-11geo1}
The variety of complex $3$-dimensional $\big(-1,1\big)$-superalgebras of type $(1,2)$  has 
dimension  $4$   and it has  $12$  irreducible components defined by  
\begin{center}
$\mathcal{C}_1=\overline{\mathcal{O}( {\rm R}_{01}^3)},$ \
$\mathcal{C}_2=\overline{\mathcal{O}( {\rm R}_{04}^\alpha)},$ \
$\mathcal{C}_3=\overline{\mathcal{O}( {\rm R}_{09})},$ \
$\mathcal{C}_4=\overline{\mathcal{O}( {\rm R}_{10})},$ \
$\mathcal{C}_5=\overline{\mathcal{O}( {\rm R}_{14})},$ \
$\mathcal{C}_6=\overline{\mathcal{O}( {\rm R}_{15})},$ \\
$\mathcal{C}_7=\overline{\mathcal{O}( {\rm R}_{18})},$ \ 
$\mathcal{C}_8=\overline{\mathcal{O}( {\rm R}_{19})},$ \
$\mathcal{C}_9=\overline{\mathcal{O}( {\rm R}_{20})},$ \
$\mathcal{C}_{10}=\overline{\mathcal{O}( {\rm R}_{22})},$  \ 
$\mathcal{C}_{11}=\overline{\mathcal{O}( {\rm R}_{23})},$  and
$\mathcal{C}_{12}=\overline{\mathcal{O}( {\rm R}_{24})}.$ \

 \end{center}
In particular, there are only $11$ rigid superalgebras in this variety.
 
\end{theorem}

\begin{proof}
Taking results from Theorem \ref{assgeo1},  we have to mention only three new degenerations
\begin{center}
${\rm R}_{14} \ {\xrightarrow{ \big(    
\frac{t}{\alpha+1}e_1,\ 
2 (\alpha-1)f_1+\alpha^2-1f_2,\  
tf_1+tf_2 \big)}} \ {\rm R}_{06}^\alpha;$  \ 
${\rm R}_{10} \ {\xrightarrow{ \big(    
 e_1,\ 
tf_1,\  
f_2 \big)}} \ {\rm R}_{11};$ \ 
${\rm R}_{14} \ {\xrightarrow{ \big(    
 e_1,\ 
f_1,\  
tf_2 \big)}} \ {\rm R}_{13}.$

\end{center}
 
\noindent
All non-degenerations follow from the following observations.
\begin{enumerate}

    \item[$\bullet$] 
$\big\{{\rm R}_{10},   {\rm R}_{14} \big\}   \not\to 
  {\rm R}_{01}^3,$  since $\big(({\rm R}_{10})_1\big)^2= \big(( {\rm R}_{14})_1\big)^2=0.$

    \item[$\bullet$] 
$ {\rm R}_{01}^3    \not\to 
\big\{    {\rm R}_{19},  {\rm R}_{22}, 
{\rm R}_{23},  {\rm R}_{24} \big\},$ since 
 $\big( ({\rm R}_{01}^3 )_0\big)^2=0.$

  \item[$\bullet$] 
$ {\rm R}_{10}    \not\to 
  \left\{ 
\begin{array}{l}
{\rm R}_{09},   {\rm R}_{15}, 
  {\rm R}_{19},  {\rm R}_{20},\\ {\rm R}_{22}, 
{\rm R}_{23},  {\rm R}_{24} 
\end{array}\right\},$ since 
 ${\rm R}_{10}$   satisfies  
 $\big\{ c_{11}^1=c_{21}^2, \ c_{12}^1=c_{13}^3=c_{31}^3=0 \big\}.$

  \item[$\bullet$] 
$ {\rm R}_{14}    \not\to 
  \left\{ 
\begin{array}{l}{\rm R}_{09},   {\rm R}_{15}, 
  {\rm R}_{19},  {\rm R}_{20},\\ {\rm R}_{22}, 
{\rm R}_{23},  {\rm R}_{24} 
\end{array}\right\},$ since 
 ${\rm R}_{14}$   satisfies  
 $\big\{ c_{11}^1=c_{21}^2, \  c_{12}^1=c_{13}^3=c_{31}^3=0 \big\}.$

\end{enumerate}

\noindent The dimensions of orbit closures are given below:
\begin{longtable}{lclclclclclclclclcl}
 
${\rm dim} \ \overline{\mathcal{O}( {\rm R}_{04}^{\alpha})} $& $=$&
${\rm dim} \ \overline{\mathcal{O}( {\rm R}_{10})}  $&$ = $& 
${\rm dim} \ \overline{\mathcal{O}( {\rm R}_{14})}  $&$ = $& 
${\rm dim} \ \overline{\mathcal{O}( {\rm R}_{18})}  $&$ = $& 
 $4;$\\ 

 ${\rm dim} \ \overline{\mathcal{O}( {\rm R}_{01}^{3})} $& $=$&
 ${\rm dim} \ \overline{\mathcal{O}( {\rm R}_{09})}  $&$ = $&
${\rm dim} \ \overline{\mathcal{O}( {\rm R}_{15})}  $&$ = $&
${\rm dim} \ \overline{\mathcal{O}( {\rm R}_{20})}  $&$ = $& $3;$\\

${\rm dim} \ \overline{\mathcal{O}( {\rm R}_{19})}  $&$ = $&
${\rm dim} \ \overline{\mathcal{O}( {\rm R}_{22})}  $&$ = $&
${\rm dim} \ \overline{\mathcal{O}( {\rm R}_{23})}  $&$ = $&
${\rm dim} \ \overline{\mathcal{O}( {\rm R}_{24})}  $&$ = $& $1.$\\
\end{longtable}

\end{proof}

\begin{theorem}\label{-11geo2}
The variety of complex $3$-dimensional $\big(-1,1\big)$-superalgebras of type $(2,1)$  has 
dimension  $5$   and it has  $13$  irreducible components defined by  
\begin{center}
$\mathcal{C}_1=\overline{\mathcal{O}( {\bf R}_{01})},$ \
$\mathcal{C}_2=\overline{\mathcal{O}( {\bf R}_{03})},$ \
$\mathcal{C}_3=\overline{\mathcal{O}( {\bf R}_{04})},$ \
$\mathcal{C}_4=\overline{\mathcal{O}( {\bf R}_{05})},$ \
$\mathcal{C}_5=\overline{\mathcal{O}( {\bf R}_{07})},$ \\
$\mathcal{C}_6=\overline{\mathcal{O}( {\bf R}_{23})},$ \
$\mathcal{C}_7=\overline{\mathcal{O}( {\bf R}_{25})},$ \
$\mathcal{C}_8=\overline{\mathcal{O}( {\bf R}_{26})},$ \
$\mathcal{C}_9=\overline{\mathcal{O}( {\bf R}_{29})},$ \\
$\mathcal{C}_{10}=\overline{\mathcal{O}( {\bf R}_{31})},$ \
$\mathcal{C}_{11}=\overline{\mathcal{O}( {\bf R}_{32})},$ \ 
$\mathcal{C}_{12}=\overline{\mathcal{O}( {\bf R}_{33})},$ \ and 
$\mathcal{C}_{13}=\overline{\mathcal{O}( {\bf R}_{36})}.$ \\
 
 \end{center}
In particular, there are only $13$ rigid superalgebras in this variety.
 
\end{theorem}

\subsubsection{Binary $\big(-1,1\big)$-superalgebras}

\begin{theorem}\label{b-11geo1}
The variety of complex $3$-dimensional binary $\big(-1,1\big)$-superalgebras of type $(1,2)$  has 
dimension  $5$   and it has  $12$  irreducible components defined by  
\begin{center}
$\mathcal{C}_1=\overline{\mathcal{O}( {\rm R}_{01}^\alpha)},$ \
$\mathcal{C}_2=\overline{\mathcal{O}( {\rm R}_{08})},$ \
$\mathcal{C}_3=\overline{\mathcal{O}( {\rm R}_{10})},$ \
$\mathcal{C}_4=\overline{\mathcal{O}( {\rm R}_{12})},$ \
$\mathcal{C}_5=\overline{\mathcal{O}( {\rm R}_{14})},$ \
$\mathcal{C}_6=\overline{\mathcal{O}( {\rm R}_{15})},$ \\
$\mathcal{C}_7=\overline{\mathcal{O}( {\rm R}_{18})},$ \ 
$\mathcal{C}_8=\overline{\mathcal{O}( {\rm R}_{19})},$ \
$\mathcal{C}_9=\overline{\mathcal{O}( {\rm R}_{20})},$ \
$\mathcal{C}_{10}=\overline{\mathcal{O}( {\rm R}_{22})},$  \ 
$\mathcal{C}_{11}=\overline{\mathcal{O}( {\rm R}_{23})},$  and
$\mathcal{C}_{12}=\overline{\mathcal{O}( {\rm R}_{24})}.$ \

 \end{center}
In particular, there are only $11$ rigid superalgebras in this variety.
 
\end{theorem}

\begin{proof}
Taking results from Theorem \ref{-11geo1},  we have to mention only five  new degenerations
\begin{center} 
${\rm R}_{01}^{\frac{t+6}{2-t}} \ {\xrightarrow{ \big(    
t\sqrt{4-2 t} e_1,\ 
\sqrt{1-\frac{t}{2}}f_1+2f_2,\  
-\frac{t}{2} \sqrt{1-\frac{t}{2}} f_1+tf_2 \big)}} \ {\rm R}_{02};$ \ 
${\rm R}_{01}^{\frac{2}{t^2}-1} \ {\xrightarrow{ \big(    
  e_1,\ 
\frac{t}{\sqrt{2}}f_1+ f_2,\  
\frac{t^2}{1-2 t^2}f_1+\frac{\sqrt{2} \left(t^2-1\right)}{t \left(2 t^2-1\right)}f_2\big)}} \ {\rm R}_{03};$ \  
${\rm R}_{01}^{\frac{\alpha+3}{1-\alpha}} \ {\xrightarrow{ \big(    
te_1,\ 
tf_1,\  
 f_2\big)}} \ {\rm R}_{04}^\alpha;$ \  
${\rm R}_{12}  \ {\xrightarrow{ \big(    
  e_1,\ 
t f_1,\  
  f_2\big)}} \ {\rm R}_{09};$ \  
${\rm R}_{01}^{\frac{2}{t}-1} \ {\xrightarrow{ \big(    
  e_1,\ 
t^2f_1+\frac{1}{2 t^2-4 t^3}f_2,\  
2 t^5 (2 t-1)f_1+tf_2\big)}} \ {\rm R}_{28}.$

\end{center}
 
\noindent
All non-degenerations follow from the following observations.
\begin{enumerate}

    \item[$\bullet$] 
${\rm R}_{01}^\alpha   \not\to 
  {\rm R}_{08},$  since ${\rm R}_{01}^\alpha$ satisfies 
  $ \big\{ c_{12}^3 c_{23}^1 = c_{21}^3 (c_{32}^1 - c_{23}^1)\big\}.$

    \item[$\bullet$] 
$\big\{{\rm R}_{01}^\alpha, \ {\rm R}_{08} \big\} 
   \not\to  \left\{ 
   \begin{array}{l}
   {\rm R}_{10}, {\rm R}_{12}, {\rm R}_{14}, {\rm R}_{15},{\rm R}_{18},\\ {\rm R}_{19}, {\rm R}_{20}, {\rm R}_{22}, {\rm R}_{23}, {\rm R}_{24}
   \end{array} \right\},$
   since 
$\big(({\rm R}_{01}^\alpha)_0 \big)^2 = \big(( {\rm R}_{08})_0\big)^2 =0.$

     \item[$\bullet$] 
$ {\rm R}_{12} 
   \not\to  \big\{  {\rm R}_{15},  {\rm R}_{19}, {\rm R}_{20}, {\rm R}_{22}, {\rm R}_{23}, {\rm R}_{24}\big\},$
   since ${\rm R}_{12}$ satisfies 
$\big\{ c_{11}^1=c_{12}^2, \ c_{21}^2=c_{13}^3=c_{31}^3=0 \big\}.$ 

\end{enumerate}

\noindent The dimensions of orbit closures are given below:
\begin{longtable}{lclclclclclclclclcl}

&&&&&&& &
${\rm dim} \ \overline{\mathcal{O}( {\rm R}_{01}^\alpha)}  $&$ = $& 
 $5;$\\ 

${\rm dim} \ \overline{\mathcal{O}( {\rm R}_{08})}  $&$ = $& 
${\rm dim} \ \overline{\mathcal{O}( {\rm R}_{10})}  $&$ = $& 
${\rm dim} \ \overline{\mathcal{O}( {\rm R}_{12})}  $&$ = $& 
${\rm dim} \ \overline{\mathcal{O}( {\rm R}_{14})}  $&$ = $& 
${\rm dim} \ \overline{\mathcal{O}( {\rm R}_{18})}  $&$ = $& 
 $4;$\\ 

&&  &  &
  & &
${\rm dim} \ \overline{\mathcal{O}( {\rm R}_{15})}  $&$ = $&
${\rm dim} \ \overline{\mathcal{O}( {\rm R}_{20})}  $&$ = $& $3;$\\

&&${\rm dim} \ \overline{\mathcal{O}( {\rm R}_{19})}  $&$ = $&
${\rm dim} \ \overline{\mathcal{O}( {\rm R}_{22})}  $&$ = $&
${\rm dim} \ \overline{\mathcal{O}( {\rm R}_{23})}  $&$ = $&
${\rm dim} \ \overline{\mathcal{O}( {\rm R}_{24})}  $&$ = $& $1.$\\
\end{longtable}

\end{proof}

\begin{theorem}\label{b-11geo2}
The variety of complex $3$-dimensional binary $\big(-1,1\big)$-superalgebras of type $(2,1)$  has 
dimension  $5$   and it has  $15$  irreducible components defined by  
\begin{center}
$\mathcal{C}_1=\overline{\mathcal{O}( {\bf R}_{01})},$ \
$\mathcal{C}_2=\overline{\mathcal{O}( {\bf R}_{03})},$ \
$\mathcal{C}_3=\overline{\mathcal{O}( {\bf R}_{04})},$ \
$\mathcal{C}_4=\overline{\mathcal{O}( {\bf R}_{06})},$ \
$\mathcal{C}_5=\overline{\mathcal{O}( {\bf R}_{08})},$ \\
$\mathcal{C}_6=\overline{\mathcal{O}( {\bf R}_{12})},$ \
$\mathcal{C}_7=\overline{\mathcal{O}( {\bf R}_{14})},$ \
$\mathcal{C}_8=\overline{\mathcal{O}( {\bf R}_{23})},$ \
$\mathcal{C}_9=\overline{\mathcal{O}( {\bf R}_{25})},$ \
$\mathcal{C}_{10}=\overline{\mathcal{O}( {\bf R}_{28})},$ \\
$\mathcal{C}_{11}=\overline{\mathcal{O}( {\bf R}_{30})},$ \
$\mathcal{C}_{12}=\overline{\mathcal{O}( {\bf R}_{31})},$ \
$\mathcal{C}_{13}=\overline{\mathcal{O}( {\bf R}_{32})},$ \ 
$\mathcal{C}_{14}=\overline{\mathcal{O}( {\bf R}_{33})},$ \ and 
$\mathcal{C}_{15}=\overline{\mathcal{O}( {\bf R}_{36})}.$ 
 
 \end{center}
In particular, there are only $15$ rigid superalgebras in this variety.
 
\end{theorem}

\begin{proof}
Taking results from Theorem \ref{altgeo2},  we have to mention only three new degenerations
\begin{center} 
${\bf R}_{06}
{\xrightarrow{ \big(    
e_1,\ 
e_2 ,\  
t f_1 \big)}}  
{\bf R}_{05};$ \    
${\bf R}_{08}
{\xrightarrow{ \big(    
e_1,\ 
e_2 ,\  
t f_1 \big)}}  
{\bf R}_{07};$ \  
${\bf R}_{06}
{\xrightarrow{ \big(    
e_1+e_2,\ 
-t^2e_2 ,\  
t f_1 \big)}}  
{\bf R}_{19}.$ \

\end{center}
\noindent
All non-degenerations follow from the following observations.
\begin{enumerate}
    \item[$\bullet$] 
${\bf R}_{06}   \not\to 
\big\{ 
{\bf R}_{01},
{\bf R}_{04},
{\bf R}_{12},
{\bf R}_{14},
{\bf R}_{23},
{\bf R}_{25},
{\bf R}_{28},
{\bf R}_{30},
 {\bf R}_{31}, {\bf R}_{32}, {\bf R}_{33}, {\bf R}_{36} \big\},$ since 
$({\bf R}_{06})_0$ is commutative and 
${\bf R}_{06}$ satisfies 
$\big\{ c_{11}^1=c_{13}^3, \ 
c_{12}^2=c_{21}^2, \ 
c_{31}^3=c_{32}^3=0 \ 
\big\}.$  

     \item[$\bullet$] 
${\bf R}_{08}   \not\to 
\big\{ 
{\bf R}_{01},
{\bf R}_{04},
{\bf R}_{12},
{\bf R}_{14},
{\bf R}_{23},
{\bf R}_{25},
{\bf R}_{28},
{\bf R}_{30},
 {\bf R}_{31}, {\bf R}_{32}, {\bf R}_{33}, {\bf R}_{36} \big\},$ since 
$({\bf R}_{06})_0$ is commutative and 
${\bf R}_{06}$ satisfies 
$\big\{ c_{11}^1=c_{31}^3, \ 
c_{12}^2=c_{21}^2=c_{13}^3, \ 
c_{22}^2=c_{23}^3, \ 
 c_{32}^3=0 \ 
\big\}.$

 \item[$\bullet$] 
$\big\{{\bf R}_{12}, {\bf R}_{14} \big\}    \not\to 
\left\{ 
\begin{array}{l}
{\bf R}_{23}, {\bf R}_{25},
{\bf R}_{28}, {\bf R}_{30}, \\
 {\bf R}_{31}, {\bf R}_{32}, {\bf R}_{33}, {\bf R}_{36}
 \end{array}\right\},$ since 
${\rm dim} \big(({\bf R}_{12})_0\big)^2 = 
{\rm dim} \big(({\bf R}_{14})_0\big)^2 = 1.$  

\end{enumerate}

\noindent The dimensions of orbit closures are given below:
\begin{longtable}{lclclclclclclclclcl}
&&
${\rm dim} \ \overline{\mathcal{O}( {\bf R}_{03})} $& $=$& 
${\rm dim} \ \overline{\mathcal{O}( {\bf R}_{06})} $& $=$& 
${\rm dim} \ \overline{\mathcal{O}( {\bf R}_{08})} $& $=$& 
 $5;$\\ 
${\rm dim} \ \overline{\mathcal{O}( {\bf R}_{01})} $& $=$& 
${\rm dim} \ \overline{\mathcal{O}( {\bf R}_{04})} $& $=$& 
${\rm dim} \ \overline{\mathcal{O}( {\bf R}_{12})} $& $=$& 
${\rm dim} \ \overline{\mathcal{O}( {\bf R}_{014})} $& $=$& 
 $4;$\\ 

&&&&${\rm dim} \ \overline{\mathcal{O}( {\bf R}_{28})} $& $=$& 
${\rm dim} \ \overline{\mathcal{O}( {\bf R}_{30})} $& $=$&
$3;$\\

&&${\rm dim} \ \overline{\mathcal{O}( {\bf R}_{23})} $& $=$& 
${\rm dim} \ \overline{\mathcal{O}( {\bf R}_{25})} $& $=$& 
${\rm dim} \ \overline{\mathcal{O}( {\bf R}_{31})} $& $=$\\ 
&&
${\rm dim} \ \overline{\mathcal{O}( {\bf R}_{32})} $& $=$& 
${\rm dim} \ \overline{\mathcal{O}( {\bf R}_{33})} $& $=$& 
${\rm dim} \ \overline{\mathcal{O}( {\bf R}_{36})} $& $=$& 
 $2.$\\ 

\end{longtable}

\end{proof}

\subsubsection{Right alternative  superalgebras}

\begin{theoremG1} 
The variety of complex $3$-dimensional right alternative superalgebras of type $(1,2)$  has 
dimension  $5$   and it has  $13$  irreducible components defined by  
\begin{center}
$\mathcal{C}_1=\overline{\mathcal{O}( {\rm R}_{01}^\alpha)},$ \
$\mathcal{C}_2=\overline{\mathcal{O}( {\rm R}_{07}^\alpha)},$ \
$\mathcal{C}_3=\overline{\mathcal{O}( {\rm R}_{08})},$ \
$\mathcal{C}_4=\overline{\mathcal{O}( {\rm R}_{10})},$ \
$\mathcal{C}_5=\overline{\mathcal{O}( {\rm R}_{12})},$ \\
$\mathcal{C}_6=\overline{\mathcal{O}( {\rm R}_{14})},$ \
$\mathcal{C}_7=\overline{\mathcal{O}( {\rm R}_{16})},$ \
$\mathcal{C}_8=\overline{\mathcal{O}( {\rm R}_{18})},$ \ 
$\mathcal{C}_9=\overline{\mathcal{O}( {\rm R}_{19})},$ \\
$\mathcal{C}_{10}=\overline{\mathcal{O}( {\rm R}_{21})},$ \
$\mathcal{C}_{11}=\overline{\mathcal{O}( {\rm R}_{22})},$  \ 
$\mathcal{C}_{12}=\overline{\mathcal{O}( {\rm R}_{23})},$  and
$\mathcal{C}_{13}=\overline{\mathcal{O}( {\rm R}_{24})}.$ \

 \end{center}
In particular, there are only $11$ rigid superalgebras in this variety.
 
\end{theoremG1}

\begin{proof}
Taking results from Theorem \ref{b-11geo1},  we have to mention only two  new degenerations
\begin{center} 
${\rm R}_{16}  \ {\xrightarrow{ \big(    
e_1,\ 
t  f_1,\  
 f_2 \big)}} \ {\rm R}_{15};$   \ ${\rm R}_{21}  \ {\xrightarrow{ \big(    
e_1,\ 
t  f_1,\  
 f_2 \big)}} \ {\rm R}_{20}.$  

\end{center}
 
\noindent
All non-degenerations follow from the following observations.
\begin{enumerate}

   \item[$\bullet$] 
$ {\rm R}_{07}^\alpha  
   \not\to  \big\{  {\rm R}_{19},  {\rm R}_{22}, {\rm R}_{23}, {\rm R}_{24} \big\},$
   since 
$\big(({\rm R}_{07}^\alpha)_0 \big)^2   =0.$

     \item[$\bullet$] 
$ {\rm R}_{16} 
   \not\to  \big\{     {\rm R}_{19},   {\rm R}_{22}, {\rm R}_{23}, {\rm R}_{24}\big\},$
   since ${\rm R}_{16}$ satisfies 
$\big\{ c_{11}^1=c_{12}^2=c_{21}^2=c_{31}^3, \  c_{13}^3=0 \big\}.$

 \item[$\bullet$] 
$ {\rm R}_{21} 
   \not\to  \big\{     {\rm R}_{19},   {\rm R}_{22}, {\rm R}_{23}, {\rm R}_{24}\big\},$
   since ${\rm R}_{21}$ satisfies 
$\big\{ c_{11}^1=c_{12}^2=c_{31}^3, \ c_{21}^2=c_{13}^3=0 \big\}.$

\end{enumerate}

\noindent The dimensions of orbit closures are given below:
\begin{longtable}{lclclclclclclclclcl}

 &&&&& &
${\rm dim} \ \overline{\mathcal{O}( {\rm R}_{01}^\alpha)}  $&$ = $& 
 $5;$\\ 

${\rm dim} \ \overline{\mathcal{O}( {\rm R}_{07}^\alpha)}  $&$ = $& 
${\rm dim} \ \overline{\mathcal{O}( {\rm R}_{08})}  $&$ = $& 
${\rm dim} \ \overline{\mathcal{O}( {\rm R}_{10})}  $&$ = $& 
${\rm dim} \ \overline{\mathcal{O}( {\rm R}_{12})}  $&$ = $\\  
${\rm dim} \ \overline{\mathcal{O}( {\rm R}_{14})}  $&$ = $& 
${\rm dim} \ \overline{\mathcal{O}( {\rm R}_{16})}  $&$ = $& 
${\rm dim} \ \overline{\mathcal{O}( {\rm R}_{18})}  $&$ = $& 
${\rm dim} \ \overline{\mathcal{O}( {\rm R}_{21})}  $&$ = $& 
 $4;$\\ 

  ${\rm dim} \ \overline{\mathcal{O}( {\rm R}_{19})}  $&$ = $&
${\rm dim} \ \overline{\mathcal{O}( {\rm R}_{22})}  $&$ = $&
${\rm dim} \ \overline{\mathcal{O}( {\rm R}_{23})}  $&$ = $&
${\rm dim} \ \overline{\mathcal{O}( {\rm R}_{24})}  $&$ = $& $1.$\\
\end{longtable}

\end{proof}

\begin{theoremG2} 
The variety of complex $3$-dimensional right alternative superalgebras of type $(2,1)$  has 
dimension  $5$   and it has  $15$  irreducible components defined by  
\begin{center}
$\mathcal{C}_1=\overline{\mathcal{O}( {\bf R}_{01})},$ \
$\mathcal{C}_2=\overline{\mathcal{O}( {\bf R}_{03})},$ \
$\mathcal{C}_3=\overline{\mathcal{O}( {\bf R}_{04})},$ \
$\mathcal{C}_4=\overline{\mathcal{O}( {\bf R}_{06})},$ \
$\mathcal{C}_5=\overline{\mathcal{O}( {\bf R}_{08})},$ \\
$\mathcal{C}_6=\overline{\mathcal{O}( {\bf R}_{12})},$ \
$\mathcal{C}_7=\overline{\mathcal{O}( {\bf R}_{14})},$ \
$\mathcal{C}_8=\overline{\mathcal{O}( {\bf R}_{24})},$ \
$\mathcal{C}_9=\overline{\mathcal{O}( {\bf R}_{25})},$ \
$\mathcal{C}_{10}=\overline{\mathcal{O}( {\bf R}_{27})},$ \\
$\mathcal{C}_{11}=\overline{\mathcal{O}( {\bf R}_{30})},$ \
$\mathcal{C}_{12}=\overline{\mathcal{O}( {\bf R}_{31})},$ \
$\mathcal{C}_{13}=\overline{\mathcal{O}( {\bf R}_{32})},$ \ 
$\mathcal{C}_{14}=\overline{\mathcal{O}( {\bf R}_{34})},$ \ and 
$\mathcal{C}_{15}=\overline{\mathcal{O}( {\bf R}_{36})}.$ 
 
 \end{center}
In particular, there are only $15$ rigid superalgebras in this variety.
 
\end{theoremG2}

\begin{proof}
Taking results from Theorem \ref{b-11geo2},  we have to mention only four new degenerations
\begin{center} 
${\bf R}_{24}
{\xrightarrow{ \big(    
e_1,\ 
e_2 ,\  
t f_1 \big)}}  
{\bf R}_{23};$ \
${\bf R}_{27}
{\xrightarrow{ \big(    
e_1-e_2,\ 
t^2e_2 ,\  
t f_1 \big)}}  
{\bf R}_{28};$ \
${\bf R}_{34}
{\xrightarrow{ \big(    
e_1,\ 
e_2 ,\  
t f_1 \big)}}  
{\bf R}_{33};$ \
${\bf R}_{34}
{\xrightarrow{ \big(    
e_1-e_2,\ 
t^2 e_2 ,\  
t f_1 \big)}}  
{\bf R}_{35}.$ 
\end{center}
\noindent
All non-degenerations follow from the following observations.
\begin{enumerate}
    \item[$\bullet$] 
${\bf R}_{24}   \not\to 
\left\{ 
\begin{array}{l}{\bf R}_{25},
 {\bf R}_{31},\\ {\bf R}_{32},   {\bf R}_{36}\end{array} \right\},$ since 
    $({\bf R}_{24})_0 \not\to ({\bf R}_{25})_0$ 
and $({\bf R}_{24})_0({\bf R}_{24})_1=({\bf R}_{24})_1({\bf R}_{24})_0=0.$  

     \item[$\bullet$] 
${\bf R}_{27}   \not\to 
\big\{ 
{\bf R}_{24},
{\bf R}_{25},
{\bf R}_{30},
 {\bf R}_{31}, {\bf R}_{32},  {\bf R}_{36} \big\},$ since \begin{flushright}
$({\bf R}_{27})_0 \not\to 
\big\{ 
({\bf R}_{24})_0,
({\bf R}_{30})_0,
({\bf R}_{32})_0,  
({\bf R}_{36})_0 \big\}
 $  and 
${\bf R}_{27}$ satisfies 
$\big\{ c_{11}^1=c_{13}^3, \  c_{31}^3=0 \  \big\}.$  
\end{flushright}

   \item[$\bullet$] 
${\bf R}_{34}   \not\to 
\big\{ 
{\bf R}_{24},
{\bf R}_{25},
{\bf R}_{30},
 {\bf R}_{31}, {\bf R}_{32},  {\bf R}_{36} \big\},$ since \begin{flushright}
$({\bf R}_{34})_0 \not\to 
\big\{ 
({\bf R}_{24})_0,
({\bf R}_{30})_0,
({\bf R}_{32})_0,  
({\bf R}_{36})_0 \big\}
 $  and 
${\bf R}_{34}$ satisfies 
$\big\{ c_{11}^1=c_{13}^3=c_{31}^3  \  \big\}.$  
\end{flushright}

\end{enumerate}

\noindent The dimensions of orbit closures are given below:
\begin{longtable}{lclclclclclclclclcl}
&&
${\rm dim} \ \overline{\mathcal{O}( {\bf R}_{03})} $& $=$& 
${\rm dim} \ \overline{\mathcal{O}( {\bf R}_{06})} $& $=$& 
${\rm dim} \ \overline{\mathcal{O}( {\bf R}_{08})} $& $=$& 
 $5;$\\ 

&&${\rm dim} \ \overline{\mathcal{O}( {\bf R}_{01})} $& $=$& 
${\rm dim} \ \overline{\mathcal{O}( {\bf R}_{04})} $& $=$& 
${\rm dim} \ \overline{\mathcal{O}( {\bf R}_{12})} $& $=$ \\&& 
${\rm dim} \ \overline{\mathcal{O}( {\bf R}_{14})} $& $=$& 
${\rm dim} \ \overline{\mathcal{O}( {\bf R}_{27})} $& $=$& 
${\rm dim} \ \overline{\mathcal{O}( {\bf R}_{34})} $& $=$& 
 $4;$\\ 

&&
&&${\rm dim} \ \overline{\mathcal{O}( {\bf R}_{24})} $& $=$& 
${\rm dim} \ \overline{\mathcal{O}( {\bf R}_{30})} $& $=$&
$3;$\\

${\rm dim} \ \overline{\mathcal{O}( {\bf R}_{25})} $& $=$& 
${\rm dim} \ \overline{\mathcal{O}( {\bf R}_{31})} $& $=$ & 
${\rm dim} \ \overline{\mathcal{O}( {\bf R}_{32})} $& $=$& 
${\rm dim} \ \overline{\mathcal{O}( {\bf R}_{36})} $& $=$& 
 $2.$\\ 

\end{longtable}\end{proof}



\begin{thebibliography}{9}
 
 

 \bibitem{akl}
 Abdelwahab H., 
 Kaygorodov I.,  
  Lubkov R.,  
 The algebraic and geometric classification of  
right alternative and semi-alternative  algebras,  
Journal of Algebra, 687 (2026), 792--824.



\bibitem{kz}
 Abdurasulov K.,  Khudoyberdiyev A.,  Toshtemirova F.,
    The geometric classification of nilpotent  Lie-Yamaguti, Bol
  and compatible Lie algebras, 
    Communications in Mathematics,  33 (2025), 3, 10.
 
 \bibitem{als}
 Abdurasulov K.,   Lubkov R.,  Saydaliyev A.,
The algebraic and geometric classification of Jordan superalgebras,
SEMR, 22 (2025),   1, 813--852.





\bibitem{alb2}
Albert A.,
     The structure of right alternative algebras,
     Annals of Mathematics (3), 59 (1954), 408--417.    
     


 \bibitem{ah}
Alvarez M.   A.,   Hernández I.,  
Varieties of nilpotent Lie superalgebras of dimension $\leq 5,$ 
Forum Mathematicum, 32 (2020),  3, 641--661.


 
 \bibitem{ahk}
 Alvarez M. A.,  Hernández I.,   Kaygorodov I.,  
 Degenerations of Jordan superalgebras, 
 Bulletin of the Malaysian Mathematical Sciences Society, 42 (2019), 6, 3289--3301.


 \bibitem{BK}
Bagnoli L.,   Kožić S.,   
Double Yangian and reflection algebras of the Lie superalgebra $\mathfrak{gl}_{m|n},$ 
Communications in Contemporary Mathematics, 27 (2025),  2,   2450007.

 \bibitem{B25} 
Bai C.,   Guo L.,   Zhang R.,   
Parity duality of super r-matrices via $\mathcal O$-operators and pre-Lie superalgebras, 
Mathematical Research Letters, 32 (2025), 1, 39--80.

 \bibitem{ben} Ben Hassine A.,  Chtioui T.,  Elhamdadi M.,  Mabrouk S.,  Cohomology and deformations of left-symmetric Rinehart algebras,
 Communications in Mathematics, 32 (2024),  2, 127--152.
 

\bibitem{BBE}
 Benayadi S.,   Bouarroudj S.,   Ehret Q.,  
 Left-symmetric superalgebras and Lagrangian extensions of Lie superalgebras in characteristic $2$, 
 Journal of Pure and Applied Algebra, 229 (2025), 11,  108086.


\bibitem{BM}
Bouarroudj S.,   Mashurov F.,  
On Zinbiel and Tortkara superalgebras, 
Bulletin of the Brazilian Mathematical Society (N.S.), 56 (2025),  3,  49.


\bibitem{BC99} Burde D., Steinhoff C.,
    Classification of orbit closures of $4$--dimensional complex Lie algebras,
    Journal of Algebra, 214 (1999), 2, 729--739.



 \bibitem{4}
Camacho L.,  Fernández Ouaridi A.,   Kaygorodov I.,   Navarro R., 
Zinbiel superalgebras, New York Journal of Mathematics, 29 (2023), 1341--1362.


 \bibitem{3} Cantarini N.,   Caselli F.,   Kac V.,  
 A Lie conformal superalgebra and duality of representations for ${\rm E}(4,4),$ Advances in Mathematics, 437 (2024),  109416. 


 
 \bibitem{C24} Chapman A., Levin I.,   Alternating roots of polynomials over Cayley-Dickson algebras,  Communications in Mathematics, 32 (2024),  2, 63--70. 

 \bibitem{C25}
 Chapman A.,   Vishkautsan S., 
 Roots and right factors of polynomials and left eigenvalues of matrices over Cayley-Dickson algebras,  Communications in Mathematics, 33 (2025),  3,  1.


 \bibitem{DES}
 Daza-García A.,   Elduque A.,   Sayin U.,   
 From octonions to composition superalgebras via tensor categories,  
 Revista Matemática Iberoamericana, 40 (2024),   1, 129–152.

 
 \bibitem{2}
Drensky V.,   Ismailov N.,   Mustafa M.,   Zhakhayev B.,  
Free bicommutative superalgebras, 
Journal of Algebra, 652 (2024), 158--187.


 \bibitem{1}
Elduque A.,  Etingof P.,   Kannan A.,  
From the Albert algebra to Kac's ten-dimensional Jordan superalgebra via tensor categories in characteristic 5, 
Journal of Algebra, 666 (2025), 387--414.



  \bibitem{FKS}
Fehlberg J\'{u}nior R.,    Kaygorodov I.,  Saydaliyev A.,
The  geometric classification of symmetric Leibniz algebras, Communications in Mathematics, 33 (2025),  1, 10.
 
  \bibitem{FKS25}
Fehlberg J\'{u}nior R.,    Kaygorodov I.,  Saydaliyev A.,
The complete classification of irreducible components of varieties of Jordan superalgebras, Communications in Mathematics, 33 (2025),  3, 15.

\bibitem{G62}
 Gaĭnov A.,  
 Alternative algebras of rank 3 and 4, 
 Algebra i Logika, 2 (1963),  4, 41--46.

 
\bibitem{ikp20}
Gaĭnov A., 
Binary Lie algebras of lower ranks,
Algebra i Logika, 2 (1963),   4, 21--40.



\bibitem{GRS}
Grishkov A., Rasskazova M., Shestakov I., 
Simple binary Lie and non-Lie superalgebra has solvable even part, 
Journal of Algebra, 655 (2024), 483--492.


 
\bibitem{GRH}
Grunewald F.,  O'Halloran J.,
    Varieties of nilpotent Lie algebras of dimension less than six,
    Journal of Algebra, 112 (1988), 315--325.
    
\bibitem{GRH3}
Grunewald F., O'Halloran J., 
Deformations of Lie algebras, 
Journal of Algebra, 162 (1993), 1, 210--224. 



\bibitem{G24}
 Guterman A.,  Zhilina S.,   
 On the lengths of descendingly flexible and descendingly alternative algebras, 
 Journal of Algebra, 651 (2024), 187--220. 
 

 \bibitem{hnt}
 Hua T.,  Napedenina E.,  Tvalavadze M.,
 Partially alternative algebras, 
  Communications in Mathematics, 33 (2025), 1, 9.


 \bibitem{5}
Huang A.,  Gao Y.,   Sun J.,   
Transposed Poisson superalgebra structures on twisted ${\rm N} = 1$ Block–Lie superalgebra, 
Pacific Journal of Mathematics, 335 (2025),  1, 119--161.


 \bibitem{isaev}
Isaev I.,
    Finite-dimensional right alternative algebras generating non-finite-basable varieties, 
    Algebra and Logic, 25 (1986), 2, 136--153.


 

 \bibitem{ikm}
Ismailov N.,   Kaygorodov I.,   Mustafa M., 
The algebraic and geometric classification of nilpotent right alternative algebras, 
Periodica Mathematica Hungarica, 84 (2022), 1, 18--30. 


 

  \bibitem{k23}
  Kaygorodov I.,     
Non-associative algebraic structures: classification and structure, Communications in Mathematics,  32 (2024), 3, 1--62.

\bibitem{MS}
Kaygorodov I.,  Khrypchenko M.,  Páez-Guillán P.,
The geometric classification of non-associative algebras: a survey, Communications in Mathematics, 32 (2024), 2, 185--284.

 
\bibitem{KM} 
 Khudoyberdiyev A.,   Muratova K.,   
 Solvable Leibniz superalgebras whose nilradical has the characteristic sequence $(n-1,1 |m)$ and nilindex $n+m,$ Communications in Mathematics, 32 (2024),   2, 27--54.

   \bibitem{KS}
Kunanbayev A.,
Sartayev B.,
Binary perm algebras and alternative algebras, 
Communications in Algebra, 2025, 
54 (2026),  1, 299--307.

    \bibitem{AA} Lopatin A.,  Rybalov A.,
 On polynomial equations over split octonions,  Communications in Mathematics,  33 (2025), 3, 8.

 
   \bibitem{l24}
 Lopes S.,
Noncommutative algebra and representation theory: symmetry, structure \& invariants, Communications in Mathematics,  32 (2024), 2, 63--117.


\bibitem{mikh}
Miheev I., 
    Simple right alternative rings,
    Algebra and Logic, 16 (1977), 6, 682--710.

     



\bibitem{MPS}
Murakami L.,  Pchelintsev S.,  Shashkov O.,   
Right alternative unital bimodules over the matrix algebras of order $\geq 3,$ 
Siberian Mathematical Journal, 63 (2022),  4, 743--757.

 

\bibitem{ser76}
    Pchelintsev S., 
    The locally nilpotent radical in certain classes of right alternative rings,
    Siberian Mathematical Journal, 17 (1976),  2, 340--360

 \bibitem{ser76bi}
Pchelintsev S., Identities defining a certain variety of right-alternative algebras, Mathematical Notes, 20 (1976), 2, 651--659.


\bibitem{sergey}
    Pchelintsev S., Right alternative Malcev-admissible nil 
    algebras of bounded index,
    Siberian Mathematical Journal, 35 (1994), 4, 754--759.



\bibitem{ser13}
    Pchelintsev S.,
    On the commutator nilpotency step of strictly $\big(-1,1\big)$-algebras, 
    Mathematical Notes, 93 (2013), 5-6, 756--762.



 


\bibitem{seroleg}
    Pchelintsev S.,  Shashkov O., 
    Simple right alternative superalgebras, 
    Journal of Mathematical Sciences (N.Y.), 284 (2024),  4, 527--544.


\bibitem{PSS}
    Pchelintsev S.,  Shashkov O.,  Shestakov I.,   
    Right alternative bimodules over Cayley algebra and coordinatization theorem, 
    Journal of Algebra, 572 (2021), 111--128.

 

 



\bibitem{skos3}
Skosyrskiy V.,
    Nilpotency in Jordan and right alternative algebras,
    Algebra and Logic, 18 (1979),  1, 73–85, 122–123.


\bibitem{skos1}
Skosyrskiy V., 
    Right alternative algebras with minimality condition for right ideals, 
    Algebra and Logic, 24 (1985), 2, 205–210.

 


 \bibitem{thedy77}
    Thedy A., 
    Nil-semisimple right alternative algebras,
    Journal of Algebra, 48 (1977), 2, 390--400.

\bibitem{thedy78}
    Thedy A.,
    Right alternative algebras and Wedderburn principal theorem,
    Proceedings of the American Mathematical Society, 72 (1978), 3, 427--435.

 

\bibitem{ualbay85}
    Umirbaev U., 
    The word problem for Jordan and right-alternative algebras, 
    Some questions and problems in analysis and algebra, 120--127, 
    Novosibirsk. Gos. Univ., Novosibirsk, 1985.
 

 

\end{thebibliography}
\end{document}